\documentclass[runningheads]{llncs}

\usepackage[T1]{fontenc}
\usepackage{amsmath, mathtools}
\usepackage{amssymb}
\usepackage{graphicx}
\usepackage{hyperref}
\usepackage{cleveref}

\begin{document}

\title{On lower bounds of the density of planar periodic sets without unit distances}

\author{Alexander Tolmachev\inst{1, 2, 3} }

\authorrunning{A. Tolmachev}

\institute{Moscow Institute of Physics and Technology, Moscow, Russia \and
Caucasus Mathematical Center of Adyghe State University, Maikop, Russia
\and
Skolkovo Institute of Science and Technology, Moscow, Russia\\
\email{tolmachev.ad@phystech.edu}}

\maketitle         
\begin{abstract}
Determining the maximal density $m_1(\bbbr^2)$ of planar sets without unit distances is a fundamental problem in combinatorial geometry. This paper investigates lower bounds for this quantity. We introduce a novel approach to estimating $m_1(\bbbr^2)$ by reformulating the problem as a Maximal Independent Set (MIS) problem on graphs constructed from the flat torus, focusing on periodic sets with respect to two non-collinear vectors. Our experimental results, supported by theoretical justifications of the proposed method, demonstrate that for a sufficiently wide range of parameters, this approach does not improve the known lower bound $0.22936 \le m_1(\bbbr^2)$. The best discrete sets found are approximations of Croft's construction. In addition, several open-source software packages for the MIS problem are compared on this task.

\keywords{Distance-avoiding sets \and independent set search \and planar colorings \and combinatorial optimization \and discrete optimization }
\end{abstract}

\section{Introduction}

In the 1960s, Leo Moser \cite{croft1967incidence,erdos1985} raised a natural question in the combinatorial geometry area: How dense can a set in $\bbbr^d$ be if it contains no pairs of points being unit-distance apart? This problem, belonging to the field of combinatorial geometry known as the combinatorics of unit distances, is closely related to the well-known Hadwiger–Nelson problem \cite{Soifer2016TheHP}, proposed by Nelson in 1950. Nelson asked: What is the minimum number of colors required to color the Euclidean plane $\bbbr^d$ such that no two points at unit distance apart are identically colored? This question has gained renewed interest in recent years following de Grey's groundbreaking discovery \cite{degrey2018chromatic} that the plane is not 4-colorable. The PolyMath16 project \cite{PolyMath16} has also contributed significantly to ongoing research in this area. Notably, any proper plane coloring requires that each color class avoids unit distances. This observation motivates the focus of this paper.

\begin{definition}
A graph with vertices in $\bbbr^d$ is called a \textbf{unit distance graph} if two vertices are adjacent if and only if they are at Euclidean distance 1 apart.
\end{definition}

The unit distance graph of $\bbbr^d$ is obtained by taking all points of $\bbbr^d$ as vertices, and connecting two vertices by an edge if and only if the corresponding points are at unit distance apart. The chromatic number of this graph is denoted by $\chi(\bbbr^d)$.

\begin{definition}
A set $A \in \bbbr^d$ is called \textbf{1-avoiding} if $A$ is the independent set in some unit distance graph with vertices in $\bbbr^d$.
\end{definition}

Assume that $A$ is measurable. The upper density of $A$, denoted by $\overline{\delta(A)}$, is defined as
$$\overline{\delta(A)} = \displaystyle\lim\sup_{R \rightarrow \infty}\frac{\lambda_d(A \cap B_d(x, R))}{\lambda_d(B_d(x, R))},$$ 
 where $\lambda_d$ is the $d$-dimensional Lebesgue measure, $B_d(x, R)$ indicates the $d$-dimensional ball of radius $R$ centered at $x$, and
 $x \in \bbbr^d$ is fixed. The definition is valid since the upper density is known to be independent of the choice of $x \in \bbbr^d$.

 Let $m_1(\bbbr^d)$ denote the supremum of the upper densities of all 1-avoiding measurable sets in $\bbbr^d$:

 $$m_1(\bbbr^d) = \sup \left\{  \overline{\delta(A)}: A \in \bbbr^d \textit{ is 1-avoiding and measurable} \right\}.$$

Estimating $m_1(\bbbr^2)$ was posed and popularized by Leo Moser (see Problem 25 in \cite{Moser1966}) and Paul Erdős. For over 50 years, various researchers have tried to improve the upper and lower bounds of this value.

The simplest non-trivial lower bound can be obtained by placing open circular discs of radius $1/2$ on a regular hexagonal lattice generated by two vectors of length $2$ forming an angle $\pi/3$. This configuration creates a set that avoids distances of $1$, with a density of $\pi/\left(8\sqrt{3} \right) \approx 0.2267$. A slight improvement upon this lower bound was proposed by Croft \cite{croft1967incidence} in 1967, establishing that $m_1(\bbbr^2) \ge 0.22936...$. 

Croft's construction is based on the shape of a ``tortoise'', which is the intersection of an open disc of radius $1/2$ and an open regular hexagon of height $x < 1$. By placing copies of this tortoise at each point of a regular hexagonal lattice with basis vectors of length $1 + x$, a 1-avoiding planar set is created. Optimizing the value of $x$ reveals that the maximal density is achieved with $x^* = 0.96533...$. It is worth noting that this construction is periodic. Despite its apparent simplicity, no further improvements to the $m_1(\bbbr^2)$ lower bound have been suggested since 1967.

However, upper bound of $m_1(\bbbr^2)$ has been improved several times over these years. Regarding the maximal density in the whole plane, Paul Erdős stated in \cite{erdos1985}: ``It seems very likely that $m_1(\bbbr^2)$ is less than 1/4''. Surprisingly, this conjecture was proven after almost 40 years in 2022. Gergely Ambrus and Mate Matolcsi \cite{ambrus2022}, using a combination of Fourier analytic and linear programming methods, first proved that $m_1(\bbbr^2) \le 0.25442$. One year later, in 2023, these authors, along with colleagues, improved their estimate to $m_1(\bbbr^2) \le 0.247$, confirming Erdős's statement \cite{ambrus2023}. Additionally, a recent paper \cite{cohen2024} explores clustering in 1-avoiding planar sets via the linear programming approach and Fourier analysis tools previously used for upper bound proofs of $m_1(\bbbr^2)$. This highlights that the study of such sets remains attractive and relevant to various scientists. 

Finally, according to the best existing bounds, $0.2293 \le m_1(\bbbr^2) \le 0.2470$, a significant gap between upper and lower estimates of $m_1(\bbbr^2)$ remains. The periodic structure of current constructions for the lower bound, coupled with the lack of lower bound improvements for over 50 years, motivates the exploration of 1-avoiding planar sets from the perspective of searching for optimal or sub-optimal periodic structures. This is the main goal of this paper, which proposes a parameterization of planar periodic sets and its analysis. More precisely, we study planar sets that are periodic with respect to two non-collinear vectors and explore the lower bound of $m_1(\bbbr^2)$ over this class of measurable planar sets without unit distances.

\section{Periodic grids}

Let's consider the measurable set $A \subset \bbbr^2$ as a coloring of the plane with two colors. The periodic plane structure can be represented by a regular lattice generated by two non-collinear vectors $\vec{v}_1$ and $\vec{v}_2$. In other words, this lattice consists of copies of parallelograms spanned by $\vec{v}_1$ and $\vec{v}_2$ that cover the entire plane.

Next, it is important to note that this parallelogram can be viewed as a flat torus generated by vectors $\vec{v}_1$, $\vec{v}_2$. This observation leads to the following definition:

\begin{definition}
    Let $T_{l_1, l_2, \alpha}$, where $l_1, l_2 \in \bbbr_+$, $\alpha \in (0, \pi/2]$, denote the flat torus defined as the parallelogram generated by vectors $\vec{v}_1, \vec{v}_2$ such that $l_1 = |\vec{v}_1|$, $l_2 = |\vec{v}_2|$, and $\angle(\vec{v}_1, \vec{v}_2) = \alpha$.
\end{definition}

The corresponding metric $\rho_{l_1, l_2, \alpha}$ between points on this flat torus is obtained by applying the Euclidean metric on the plane. 

\begin{definition}
    Let $p_1 = x_1 \vec{v}_1 + y_1 \vec{v}_2, p_2 = x_2 \vec{v}_1 + y_2 \vec{v}_2$, where $x_1, y_1, x_2, y_2 \in [0, 1)^2$, be arbitrary points on the flat torus $T_{l_1, l_2, \alpha}$. Then, the corresponding metric $\rho_{l_1, l_2, \alpha}$ is defined as follows:
    $$\rho_{l_1, l_2, \alpha}\left(p_1, p_2\right) = \min \left\{ \left| (x_2 - x_1 + m) \cdot \vec{v}_1 + (y_2 - y_1 + n) \cdot \vec{v}_2 \right | : m, n \in \bbbz \right\}. $$ 
\end{definition}

The point $p = x \vec{v}_1 + y \vec{v}_2 \in T_{l_1, l_2, \alpha}$ will be denoted by its affine coordinates: $p = (x, y)$, where $x, y \in [0, 1)^2$. Metric indices $l_1, l_2, \alpha$ indicate that the metric definition depends on the corresponding flat torus parameters. However, from this point forward, we will denote the metric $\rho_{l_1, l_2, \alpha}$ simply as $\rho$, omitting the metric indices for convenience.

The following lemma highlights the possibility of computing the metric $\rho$ via a finite search (instead of infinite) over $m, n$ variables in the formula above from the metric definition:

\begin{lemma}
    Let $T_{l_1, l_2, \alpha}$ be the flat torus such that $l_1 \le l_2$. Let $p_1 = (x_1, y_1)$, $p_2 = (x_2, y_2) \in T_{l_1, l_2,\alpha}$ be arbitrary points. Then,     

    \begin{align*}
        & \rho(p_1, p_2) = \min \{ \left| (x_2 - x_1 + m) \cdot \vec{v}_1 + (y_2 - y_1 + n) \cdot \vec{v}_2 \right | : \\ & \hspace{6cm} n \in \mathbb{Z} \cap [-k, k], m \in \{ \lfloor f(n) \rfloor, \lceil f(n)  \rceil \} \}, \\
        & \text{where } k = \left\lceil\frac{1}{1 - \cos^2 \alpha}\right\rceil + 1,  f(n) = -\left(x_2 - x_1 +\frac{(y_2 - y_1 + n) \cdot l_2\cos{\alpha}}{l_1} \right).
    \end{align*}

\label{metric_computation}
\end{lemma}

\begin{proof}

Denote $t_1 = x_2 - x_1$ and $t_2 = y_2 - y_1$, where $t_1, t_2 \in (-1, 1)$. Recall that for any vector $\vec{v} \in \bbbr^2$ and the standard scalar product the equation $(\vec{v}, \vec{v}) = |\vec{v}|^2$ holds. Then, we will minimize the following function $g(m, n): \bbbr \times \bbbr \rightarrow \bbbr_+$ over integer arguments: 

\begin{align*}
g(m, n) & = \left| (x_2 - x_1 + m) \cdot \vec{v}_1 + (y_2 - y_1 + n) \cdot \vec{v}_2 \right |^2 = \\ 
        & = \left((t_1 + m) \cdot \vec{v}_1 + (t_2 + n) \cdot \vec{v}_2, (t_1 + m) \cdot \vec{v}_1 + (t_2 + n) \cdot \vec{v}_2 \right) =  \\
        & = (t_1 + m)^2 (\vec{v}_1, \vec{v}_1) + (t_2 + n)^2 (\vec{v}_2, \vec{v}_2) + 2(t_1 + m)(t_2 + n)(\vec{v}_1, \vec{v}_2) = \\
        & = l_1^2(t_1 + m)^2 + l_2^2(t_2 + n)^2 + 2(t_1 + m)(t_2 + n)\cdot l_1 l_2 \cos \alpha.
\end{align*}

The last term can be written as the quadratic function regarding the variable $m$:

\begin{align*}
 g(m, n) & = l_1^2(t_1 + m)^2 + l_2^2(t_2 + n)^2 + 2(t_1 + m)(t_2 + n)\cdot l_1 l_2 \cos \alpha = \\
         & = l_1^2 \left(m + t_1 + \frac{(t_2 + n)l_2 \cos\alpha}{l_1} \right)^2 + l_2^2 (t_2 + n)^2 (1 - \cos^2 \alpha) = \\ 
         & = l_1^2 \left(m - f(n) \right)^2 + l_2^2 (t_2 + n)^2 (1 - \cos^2 \alpha) \rightarrow \displaystyle\min_{m, n \in \bbbz}
\end{align*}

Denote by $m_{\textit{opt}}, n_{\textit{opt}} \in \bbbz$ are global minimum point of the function $g$ over integer arguments. The first observation is as follows: if the value of variable $n$ is fixed, then optimal (to minimize the function $g$) $m \in \bbbz$ is equal to one of the nearest integer values to $f(n)$, so $m_{\textit{opt}}  \in \left \{ \lfloor f(n) \rfloor, \lceil f(n)  \rceil \right \} $.

This leads to the next natural question: how to find the optimal value of the variable $n$ corresponding to the global minimum of function $g$?

Without loss of generality, assume $t_1, t_2 \ge 0$. Denote the vectors $\vec{w}, \vec{u}$ as follows:
$\vec{w} = t_1 \cdot \vec{v}_1 + t_2 \cdot \vec{v}_2$ and
$\vec{u} = (t_1 - 1) \cdot \vec{v}_1 + (t_2 - 1) \cdot \vec{v}_2$.

Then, $|\vec{w}| \le t_1 |\vec{v}_1| + t_2 |\vec{v}_2|$ and $|\vec{u}| \le (1 - t_1) |\vec{v}_1| + (1 - t_2) |\vec{v}_2|$, because $t_1, t_2 \in [0, 1)$ in this case. Hence, $|\vec{w}| + |\vec{u}| \le (t_1 + 1 - t_1)|\vec{v}_1| + (t_2 + 1 - t_2)|\vec{v}_2| = l_1 + l_2 \le 2l_2$ due to $l_1 \le l_2$. Therefore, $\min \{ |\vec{w}|, |\vec{u}| \} \le 2l_2 / 2 = l_2$. From  $g(0, 0) = |\vec{w}|^2$ and $g(-1, -1) = |\vec{u}|^2$, it follows that $\min \{ g(0, 0), g(-1, -1) \} \le l_2^2$. Since $ \left [ \rho(p_1, p_2) \right ]^2 = \min \left \{ g(m, n): m, n \in \bbbz \right \}$, we obtain that $\left [ \rho(p_1, p_2) \right ]^2 \le l_2^2$. Hence, $\rho(p_1, p_2) \le l_2$.

Cases $(t_1 \ge 0, t_2 < 0)$, $(t_1 < 0, t_2 \ge 0)$, $(t_1 < 0, t_2 < 0)$ can be considered analogously. The vector $\vec{u}$ is defined the same in all cases, while the vector $w$ is defined differently. For instance, in case $(t_1 < 0, t_2 \ge 0)$, we define $\vec{u} = (t_1 + 1) \vec{v}_1 + (t_2 - 1) \vec{v}_2$.

The proven inequality $\rho(p_1, p_2) \le l_2$ leads us to the second observation: if $g(m, \hat{n}) > l_2^2$ for all $m \in \bbbz$, then $\hat{n} \neq n_{\textit{opt}}$. 

Notice that $$g(m, n) = l_1^2 \left(m - f(n) \right)^2 + l_2^2 (t_2 + n)^2 (1 - \cos^2 \alpha) \ge l_2^2 (t_2 + n)^2 (1 - \cos^2 \alpha) > l_2^2$$ if $|t_2 + n| > \frac{1}{1 - \cos^2 \alpha}$. Since $t_2 \in (-1, 1)$, we obtain that $g(m, n) > l_2^2$ if $|n| >  \lceil\frac{1}{1 - \cos^2 \alpha}\rceil + 1$. Therefore, to find the $n_{\textit{opt}}$, it is sufficient to consider the integer values $n \in [-k, k]$, where $k = \lceil\frac{1}{1 - \cos^2 \alpha}\rceil + 1$.

Hence, we prove that to find the minimum of the function $g(m, n)$, it is sufficient to consider the integer values $n \in [-k, k]$, where $k = \lceil\frac{1}{1 - \cos^2 \alpha}\rceil + 1$, and for each value of $n$, consider the two values of the variable $m$ that can be computed via the function $f(n)$. Overall, we have proved that the minimum of $g(m, n)$ over this set of pairs $(n, m)$ is the global minimum of $g(m, n)$ over integer arguments. \qed
\end{proof} 

To improve readability, we will omit the word ``planar'' in some cases, as only 1-avoiding planar sets are considered in this paper. The following definition will be useful in the next sections to clarify the relationship between 1-avoiding planar periodic sets and sets without unit distances (via metric $\rho$) on the flat torus.

\begin{definition}
    Let us call the flat torus $T_{l_1, l_2, \alpha}$ (generated by vectors $\vec{v}_1, \vec{v}_2$) \underline{perfectly periodic} if for any two points $p_1 = (x_1, y_1), p_2 = (x_2, y_2) \in T_{l_1, l_2, \alpha}$ the following implication holds:

    $$ \rho(p_1, p_2) \ne 1 \Rightarrow \left| (x_2 - x_1 + m) \cdot \vec{v}_1 + (y_2 - y_1 + n) \cdot \vec{v}_2 \right | \ne 1 \quad \forall m, n \in \bbbz $$
    
\end{definition}

This definition means that for such a flat torus, if the set $A \subset T_{l_1, l_2, \alpha}$ is 1-avoiding via the $\rho$ metric, then it can be copied onto the regular lattice generated by vectors $\vec{v}_1, \vec{v}_2$ to form the periodic 1-avoiding set $\hat{A} \subset \bbbr^2$ via the Euclidean metric on the plane. This leads to the natural question: how to describe the set of perfectly periodic toruses? The following lemma proposes the sufficient condition for the parameters $l_1, l_2, \alpha$.

\begin{lemma}
A flat torus $T_{l_1, l_2, \alpha}$ is perfectly periodic if $l_1 \ge 2$ and $l_2 \sin \alpha \ge 2$. 

\label{lemma_perfectly_periodic}
\end{lemma}

\begin{proof}

Consider arbitrary two points with its affine coordinates $p_1 = (x_1, y_1)$, $p_2 = (x_2, y_2) \in T_{l_1, l_2, \alpha}$ and suppose that $\rho(p_1, p_2) \ne 1$. We will show that $$\left| (x_2 - x_1 + m) \cdot \vec{v}_1 + (y_2 - y_1 + n) \cdot \vec{v}_2 \right | \ne 1 \; \forall m, n \in \bbbz.$$ This will prove that $T_{l_1, l_2, \alpha}$ is perfectly periodic.

If $\rho(p_1, p_2) > 1$, then, according to the metric definition, $$\left| (x_2 - x_1 + m) \cdot \vec{v}_1 + (y_2 - y_1 + n) \cdot \vec{v}_2 \right | > 1 \; \forall m, n \in \bbbz.$$ 

If $\rho(p_1, p_2) = 0$, then $x_1 = x_2, y_1 = y_2$, because $x_1, x_2, y_1, y_2 \in [0, 1)$. Hence, for any $(m, n) \ne (0, 0)$ the following equality holds: $$\left| (x_2 - x_1 + m) \cdot \vec{v}_1 + (y_2 - y_1 + n) \cdot \vec{v}_2 \right | \ge \min \{|\vec{v}_1|, |\vec{v}_2| \}  = \min \{l_1, l_2 \} \ge 2.$$ Therefore, in the case $\rho(p_1, p_2) = 0$ the required statement has been proved.

It remains to prove this fact in the last case: $\rho(p_1, p_2) \in (0, 1)$. Let the point $O$ be the origin of the coordinates and the points $A_i$, $B_j$, where $i, j \in \bbbz$ are defined in the following way:

$\overline{OA_i} = (x_1 - x_2 - i) \vec{v}_1 \; \forall i \in \bbbz$ and $\overline{OB_j} = (y_2 - y_1 + j) \vec{v}_2 \; \forall j \in \bbbz$.

Then, $\overline{A_iB_j} = \overline{OB_j} - \overline{OA_i} = (x_2 - x_1 + i) \vec{v}_1 + (y_2 - y_1 + j) \vec{v}_2$. So, we can reformulate our statement in the following way: 
$$\displaystyle\min_{i, j \in \bbbz} A_iB_j < 1 \Rightarrow A_iB_j \ne 1 \quad \forall i, j \in \bbbz.$$

Denote by $a$ the line containing points $A_i, i \in \bbbz$. The line $b$ is defined analogously. Let $a_\perp$ be the perpendicular to the line $a$. The length of the projection of the segment $AB$ onto the line $l$ will be denoted by $\pi(AB, l)$. The distance from the point $A$ to the line $l$ will be denoted by $d(A, l)$. The construction used for this proof is shown in Fig.~\ref{lemma2_proof}.

\begin{figure}
\includegraphics[width=\textwidth]{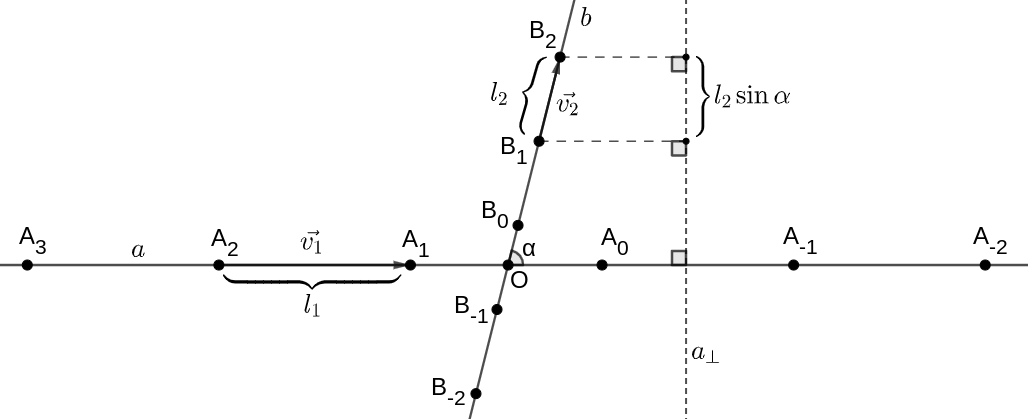}
\caption{Construction for the proof of Lemma~\ref{lemma_perfectly_periodic} in case $y_2 - y_1 \in [0, 1)$.} \label{lemma2_proof}
\end{figure}

Let $y_2 - y_1 \in [0, 1)$. Thus, if integer $j \ge 1$, then $A_iB_j \ge d(B_j, a) \ge \pi(B_0B_j, a_\perp) = j \cdot l_2 \sin \alpha \ge 2$. If integer $j \le -2$, then $A_iB_j \ge d(B_j, a) \ge \pi(B_{-1}B_j, a_\perp) = (-1 -j) \cdot l_2 \sin \alpha \ge 2$. So, for arbitrary $i \in \bbbz$ and $j \in \bbbz / \{0, -1\}$ we proved that $A_iB_j \ge 2$. If $y_2 - y_1 \in (-1, 0]$ it could be analogously shown that $A_iB_j > 2$ for arbitrary $i \in \bbbz$ and $j \in \bbbz / \{0, 1\}$.

Without loss of generality, suppose that $y_2 - y_1 \in [0, 1)$ and $\displaystyle\min_{i, j \in \bbbz} A_iB_j$ is achieved between points $A_{i_0}$ and $B_0$, so $A_{i_0}B_0 < 1$. For every integer $i \ne i_0$ we can apply the triangle inequality: $A_iB_0 + A_{i_0}B_0 > A_iA_{i_0} = |\vec{v}_1| \cdot |i_0 - i| = l_1 \cdot |i_0 - i| \ge l_1 \ge 2$. Then, $A_iB_0 \ge 2 - A_{i_0}B_0 > 1$, because $A_{i_0}B_0 < 1$. So, $A_iB_0 \ne 1 \; \forall i \in \bbbz$.

It remains to prove that $A_iB_{-1} \ne 1 \; \forall i \in \bbbz$. Suppose that $A_iB_{-1} = 1$. Then, for arbitrary integer $i \in \bbbz$ the following chain of inequalities holds: $1 = A_iB_{-1} \ge d(B_{-1}, a) = \pi(B_{-1}B_0, a_\perp) - d(B_0, a) = l_2\sin\alpha - d(B_0, a) \ge l_2\sin\alpha - A_{i_0}B_0 > l_2\sin\alpha - 1 \ge 2 - 1 = 1$. It brings to the contradiction, so $A_iB_{-1} \ne 1$ for all $i \in \bbbz$. If $\displaystyle\min_{i, j \in \bbbz} A_iB_j$ is achieved between points $A_{i_1}$ and $B_{-1}$, the proof is completely analogous. Hence, we have proved that $A_iB_{j} \ne 1 \; \forall i, j \in \bbbz$.

The case $y_2 - y_1 \in (-1, 0]$ could be considered analogously via replacing $B_{-1}$ to $B_{1}$. It finalized the proof of this lemma in the case $\rho(p_1, p_2) \in (0, 1)$.
\qed
\end{proof}

Note that if we consider the symmetric conditions $l_1 \sin \alpha \ge 2$ and $l_2 \ge 2$, then the flat torus $T_{l_1, l_2, \alpha}$ is also perfectly periodic. An important consequence of this lemma is that if the parallelogram's heights are no less than $2$, then the flat torus generated by this parallelogram is perfectly periodic.

\section{Graph construction}

Denote by $\operatorname{diam}(F)$ the diameter of the set $F \in T_{l_1,l_2, \alpha}$ with respect to the metric $\rho$. Let $T_{l_1, l_2, \alpha} = F_1 \sqcup F_2 \sqcup F_3 \sqcup \dots \sqcup F_n$ be a partitioning of the perfectly periodic torus where $\forall i \in \{1, 2, \dots , n\}$ $F_i$ is a measurable set, $\operatorname{diam}(F_i) < 1$. Let $p_1, p_2, \dots, p_n \in T_{l_1, l_2, \alpha}$ be arbitrary points such that $p_i \in F_i$ for every $i \in \{1, 2, \dots , n\}$.

Consider some undirected graph $G = (V, E)$, where the set of vertices $V = \{p_1, p_2, \dots, p_n\}$, and the set of edges $E \subset V \times V$ satisfy the following rule for all $1 \le i < j \le n$:

$$ (p_i, p_j) \notin E \Rightarrow \rho(q_1, q_2) \ne 1 \; \forall q_1 \in F_i, q_2 \in F_j.$$

Notice that the structure of this graph strongly depends on $F_1, F_2, \dots, F_n$. Next lemma presents the idea to obtain a lower bound of $m_1(\bbbr^2)$ via independent sets in these graphs.

\begin{lemma}
    Let $M \subset \{p_1, p_2, \dots, p_n\}$ be an independent set in the undirected graph $G = (V, E)$. This means that for any distinct $a, b \in M$ $(a, b) \notin E$. Let $F_i$ denote the above mentioned set associated with vertex $p_i$ for $i = 1, 2, \dots, n$. Then, 
    $$m_1(\bbbr^2) \ge \frac{\sum\limits_{p_i \in M} \lambda_2(F_i)}{\sum\limits_{j = 1}^n \lambda_2(F_j)},$$
    where $\lambda_2$ represents the 2-dimensional Lebesgue measure (i.e., area). 

\label{lemma_graph_construction}
\end{lemma}

\begin{proof}

Consider the such set $A = \bigsqcup\limits_{p_i \in M} F_i$. The Lebesgue measure of this set on the plane is given by $\lambda_2(A) = \sum\limits_{p_i \in M} \lambda_2(F_i)$. Similarly, the Lebesgue measure of the parallelogram $B$ is $\lambda_2(B) = \sum\limits_{j = 1}^n \lambda_2(F_j)$. Consider the set $\hat{A} \in \bbbr^2$ generated by replicating $A$ on a regular lattice with vectors $\vec{v}_1$, $\vec{v}_2$. Since $A$ is measurable, $\hat{A}$ is also measurable. Due to the periodic structure of $\hat{A}$ we have $\overline{\delta(\hat{A})} = \lambda_2(A) / \lambda_2(B)$.

Therefore, if we can prove that $\hat{A}$ is 1-avoiding, then it follows that $$m_1(\bbbr^2) \ge \lambda_2(A) / \lambda_2(B) = \frac{\sum\limits_{p_i \in M} \lambda_2(F_i)}{\sum\limits_{j = 1}^n \lambda_2(F_j)}.$$

Let $P, Q \in \hat{A} \subset \bbbr^2$ be such that $\rho(P, Q) = 1$. Due to the construction of $\hat{A}$, there exist points $P_0, Q_0 \in A$ such that $\overline{P_0P} = a_1\vec{v}_1 + a_2\vec{v}_2$ and $\overline{Q_0Q} = b_1\vec{v}_1 + b_2\vec{v}_2$, where $a_1, a_2, b_1, b_2 \in \bbbz$. Furthermore, let $P_0 \in F_i$ and $Q_0 \in F_j$, where $F_i, F_j \subset A$.

If $i = j$, then $\rho(P_0, Q_0) < 1$, because $\operatorname{diam}(F_i) < 1$. Since $\rho(P_0, Q_0) < 1$ and the torus is perfectly periodic, it follows that $\rho(P, Q) \ne 1$. This contradicts our initial assumption that $\rho(P, Q) = 1$.

If $i \ne j$, then $(p_i, p_j) \notin  E$, because $M$ is an independent set in the graph $G$. Therefore, by the construction of $G$, there are no two points $p \in F_i, q \in F_j$ such that $\rho(p, q) < 1$. Consequently, $\rho(P_0, Q_0) < 1$. Since the torus is perfectly periodic, this implies that $\rho(P, Q) \ne 1$. Last inequality this also leads to the contradiction with our initial assumption.

We have considered all possible cases and shown that no two points $P, Q \in \hat{A} \subset \bbbr^2$ satisfy $\rho(P, Q) = 1$. Therefore, $\hat{A}$ is 1-avoiding.
\qed
\end{proof}

This lemma allows us to reformulate the problem of estimating $m_1(\bbbr^2)$ as a problem of finding a maximal independent set in a corresponding graph. While this section has not explicitly defined the set of edges for the graph, it has described a sufficiently large class of graphs based on the condition for the set $E$. To obtain the most accurate $m_1(\bbbr^2)$ lower estimates, however, it is desirable to minimize the number of edges in these graphs. This is because a smaller number of edges leads to a larger maximum independent set size, thereby strengthening the lower bound. The precise construction of this graph are proposed and discussed in the following section.

\section{Flat torus partitioning into equal hexagons}

This section proposes a construction of the graph mentioned in the previous section for an arbitrary flat torus $T_{l_1, l_2, \alpha}$ parameterized by side lengths $l_1, l_2$ and angle $\alpha \in (0, \pi/2]$ between them.

\begin{figure}
\begin{tabular}{cc} 
    {\includegraphics[width=0.48\textwidth]{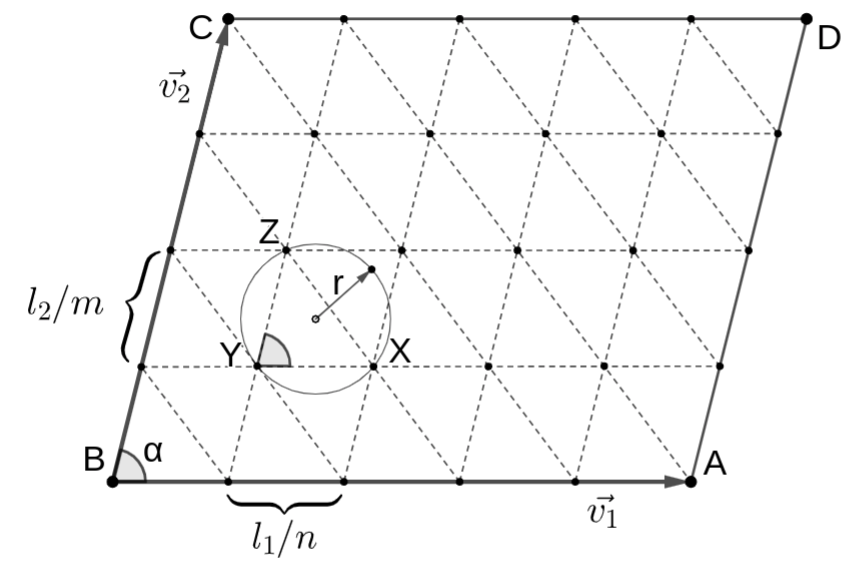}} &
    {\includegraphics[width=0.46\textwidth]{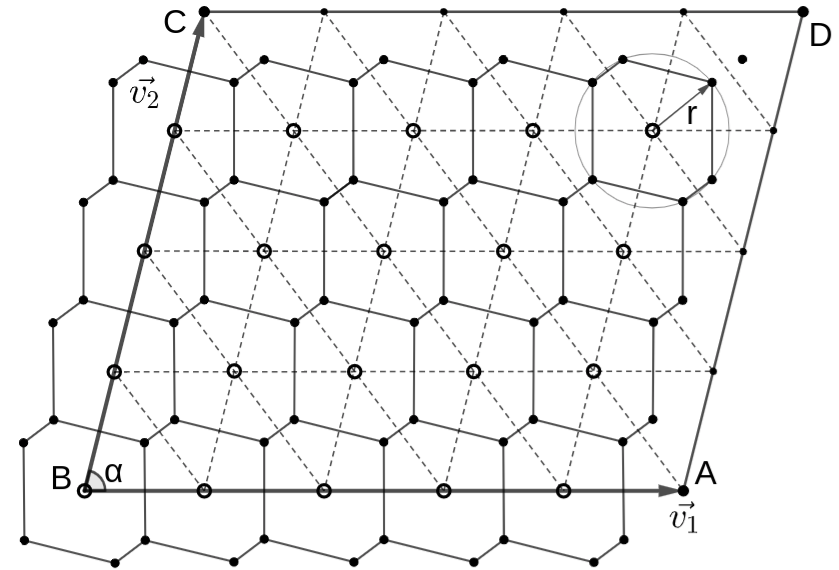}}
\end{tabular}
\caption{Graph construction concepts.} \label{graph_construction}
\end{figure}

Without loss of generality, this torus can be represented as the parallelogram $ABCD$, such that $\overline{BA} = \vec{v}_1$, $\overline{BC} = \vec{v}_2$, $|AB| = l_1$, $|BC| = l_2$ and $\angle ABC = \alpha \in (0, \pi/2]$.

We introduce parameters $n, m \in \bbbn$ to represent the grid discretization of the flat torus. It allows us to divide sides $AB$ and $BC$ into $n$ and $m$ equal parts respectively. Actually, we consider the $\varepsilon$-net in the metric space $\left(T_{l_1, l_2, \alpha}, \rho \right)$. 

Then, we define the set of graph vertices $V_{n,m}$ as follows:
$$V_{n,m} = \left\{\frac{a}{n} \vec{v}_1 + \frac{b}{m} \vec{v}_2 : a \in \{0, 1, \dots , n - 1\}, b \in \{0, 1, \dots , m - 1\}  \right\}.$$

These $|V_{n, m}| = n \cdot m$ points generate a triangulation of the parallelogram into equal triangles with sides $l_1 / n$, $l_2 / m$, and angles $\alpha$ between them (see Fig. ~\ref{graph_construction}, left part).

Consider an arbitrary triangle from this triangulation. Let $XYZ$ be one such triangle with sides $XY = a = l_1 / n$, $YZ = b = l_2 / m$ and the angle $ \angle XYZ = \alpha$ between them. It's third side can be computed via the Cosine rule: 
$XZ = c = \sqrt{a^2 + b^2 - 2ab \cos{\alpha}}$. Then, the radius $r$ of the circumcircle of $XYZ$ can be calculated via the Sine rule: $r = \frac{XZ}{2\sin\alpha} = \frac{c}{2\sin\alpha}$, because the side $XZ$ is opposite to the angle $\angle XYZ$.

Next, we consider the Voronoi diagram (see Fig.~\ref{graph_construction}, right part) constructed for the set of points $V_{n, m}$ using the $\rho$ metric. The boundaries of the Voronoi cells are segments of the perpendicular bisectors of the sides of the triangulation described above. Since $|V_{n, m}| = nm$ and all triangles in the triangulation are equal, the Voronoi sets are equal, disjoint hexagons. The center of each hexagon coincides with the center of the circumcircle of the corresponding triangle in the triangulation, and these centers are precisely the points in the set $V_{n, m}$ (see Fig.~\ref{graph_construction}, right part).

Denote these hexagons by $H_1, H_2, \dots, H_{nm}$. This construction represents the flat torus partitioning into $nm$ parts: $T_{l_1, l_2, \alpha} = H_1 \sqcup H_2 \sqcup \dots H_{nm}$. The radius of the circumcircle of each $H_i$ is exactly $r$, which can be calculated as described above, due to the geometric structure of this partitioning. The diameter of every hexagon $H_i$ is achieved on three diagonals: $\operatorname{diam}(H_i) = 2r$. Therefore, this partitioning satisfies the properties of Lemma \ref{lemma_graph_construction} if and only if $2r < 1$.

The set $E_{nm}$ of edges for the corresponding graph is defined as follows:
$$E_{n, m} = \left\{ (v_1, v_2) \in V_{n, m} \times V_{n, m}: \rho(v_1, v_2) \in [1 - 2r, 1 + 2r] \right\}.$$

Note that in this construction, an edge could be drawn if and only if $F_i \cup F_j$ contains a unit distance. Since experiments showed that this does not lead to a noticeable improvement in the estimate, the set of edges was computed using a simplified algorithm. The graph $G_{n,m}$ contains extra edges because it was constructed for covering the torus with circles, not for partitioning into hexagons.

The next lemma guarantees that constructed undirected graph $G_{n, m} = (V_{n, m}, E_{n, m})$ satisfies the graph edge construction rules described at the beginning of the previous section.

\begin{lemma}
For the graph $G_{n, m}$ defined above, if two distinct vertices $v_1, v_2$ satisfy $(v_1, v_2) \notin E_{n, m}$, then $\rho(q_1, q_2) \ne 1$ for all $q_1 \in F_1, q_2 \in F_2$, where $F_1, F_2$ are the corresponding equal hexagons with centers at $v_1, v_2$ respectively.

\label{lemma_distance_inequality}
\end{lemma}

\begin{proof}
    Consider two arbitrary distinct vertices $v_1, v_2$ of $G_{n,m}$ and suppose that $(v_1, v_2) \notin E_{n, m}$. This implies $\rho(v_1, v_2) \notin [1 - 2r, 1 + 2r]$ according to the definition of $E_{n, m}$ .
    
    Now, suppose there exist two points $q_1, q_2$ such that $q_1 \in F_1$, $q_2 \in F_2$, and $\rho(q_1, q_2) = 1$. Since $q_1 \in F_1$, we have $\rho(q_1, v_1) \le r$ because $F_1 \subset \omega(v_1, r)$, where $\omega(v_1, r)$ is the disk of radius $r$ centered at $v_1$. Analogously, $\rho(q_2, v_2) \le r$. Applying the triangle inequality to the metric $\rho$, we obtain:

    $$1 = \rho(q_1, q_2) \le \rho(q_1, v_1) + \rho(v_1, v_2) + \rho(v_2, q_2) \le 2r + \rho(v_1, v_2).$$

    Hence, $1 - 2r \le \rho(v_1, v_2)$. Similarly, applying the triangle inequality again yields:

    $$\rho(v_1, v_2) \le \rho(v_1, q_1) + \rho(q_1, q_2) + \rho(q_2, v_2) \le 2r + 1.$$

    Combining these inequalities, we have $1 - 2r \le \rho(v_1, v_2) \le 1 + 2r$. This contradicts our earlier assumption that $\rho(v_1, v_2) \notin [1 - 2r, 1 + 2r]$. Therefore, our initial supposition that such points $q_1$ and $q_2$ exist must be false. This completes the proof of this lemma.
    \qed  
\end{proof}

These lemmas lead to the main theorem of this paper, which will be crucial and useful in our experimental pipeline.

\begin{theorem}
Let $l_1, l_2 \in \bbbr_+$ be side lengths and $\alpha \in (0, \pi/2]$ be the angle parameterizing the perfectly periodic torus $T_{l_1, l_2, \alpha}$, and let $n, m \in \bbbn$, such that $2r < 1$, where $r$ is the radius of the circumcircle of triangles in the constructed torus triangulation. Next, $G_{n, m} = (V_{n, m}, E_{n, m})$ is the undirected graph constructed via the described procedure.

If $M \subset V_{n, m}$ is an independent set of the graph $G_{n, m}$, then

$$m_1(\bbbr^2) \ge \frac{|M|}{n \cdot m}.$$

\label{main_theorem}
\end{theorem}

\begin{proof}
The torus partitioning into equals hexagons $T_{l_1, l_2, \alpha} = H_1 \sqcup H_2 \sqcup \dots \sqcup H_{nm}$ implies that the areas of all hexagons are equal.

Lemma \ref{lemma_distance_inequality} proved that the graph $G_{n, m}$ described in this section satisfies the conditions of Lemma \ref{lemma_graph_construction}. Note that the diameter of $F_i$ (with respect to the metric $\rho$) for any hexagon $F_i$ is no greater than $2r$ (the length of each its diagonal via Euclidean metric). This allows us to apply Lemma \ref{lemma_graph_construction} to the graph $G_{n, m}$ and the aforementioned torus partitioning into equal hexagons, because $\operatorname{diam}(F_i) \le 2r < 1$, where last inequality follows from the conditions of this theorem.

\begin{figure}
\begin{center}
\includegraphics[width=0.5\textwidth]{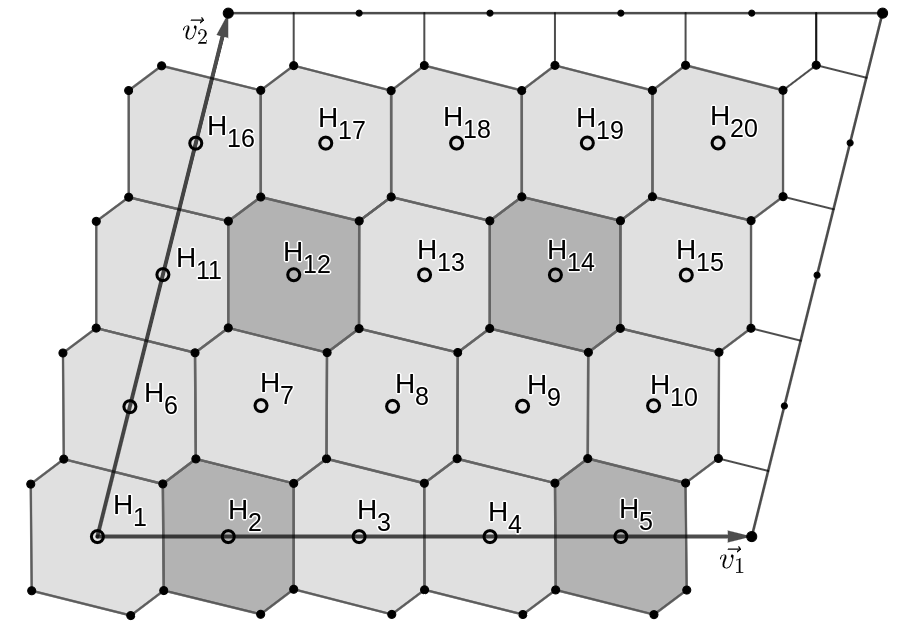}
\caption{The construction of the graph $G_{n, m}$ in case of $n = 5, m = 4$. Dark hexagons denote the independent set $M$. This example leads to the following estimation: $m_1(\bbbr^2) \ge \frac{|M|}{nm} = 4 / 20 = 0.2$.} \label{theorem_proof}
\end{center}
\end{figure}

Next, we can rewrite $\sum\limits_{i = 1}^{nm} H_{nm} = nm \cdot \lambda_2(H_1)$, because all hexagons are equal and each hexagon is measurable. Here, the Lebesgue measure $\lambda_2(H_1)$ equals to the area of the hexagon $H_1$. Since $M$ is an independent set in the graph $G$, we have $\sum\limits_{p_i \in M} \lambda_2(H_i) = |M| \cdot \lambda_2(H_1)$ due to the equality of all hexagons. Figure~\ref{theorem_proof} illustrates the main idea of this theorem. Applying the Lemma~\ref{lemma_graph_construction} to this torus partitioning, the corresponding graph and its independent set $M$, we obtain that

$$m_1(\bbbr^2) \ge \frac{\sum\limits_{p_i \in M} \lambda_2(H_i)}{\sum\limits_{j = 1}^{nm} \lambda_2(H_j)} = \frac{|M| \cdot \lambda_2(H_1)}{nm \cdot \lambda_2(H_1)} = \frac{|M|}{n \cdot m}.$$
\qed

\end{proof}

\section{Experiments \& Results}

\subsection{Experiment scheme}
According to Theorem~\ref{main_theorem}, the experiment scheme for the $m_1(\bbbr^2)$ estimation is as follows:

\begin{enumerate}
    \item Choose side lengths $l_1, l_2 \in \bbbr_+$ and angle $\alpha \in (0, \pi/2]$, such that flat torus $T_{l_1, l_2, \alpha}$ is perfectly periodic. In other words, these variable should satisfy one of following two cases: 
    \begin{enumerate}
    \item $l_1 \sin \alpha \ge 2, l_2 \ge 2$,
    \item $l_1 \ge 2, l_2 \sin \alpha \ge 2.$
    \end{enumerate}

    \item Choose integer parameters $n, m \in \bbbn$ that correspond to the torus grid discretization, such that the above mentioned circumscribed radius $r < 1/2$.

    \item Construct the graph $G_{n, m} = \left(V_{n, m}, E_{n, m} \right)$. For each pair of vertices the distance between them is computed using Lemma~\ref{metric_computation}.

    \item Find the maximal independent set $M \subset V_{n, m}$ via one of maximum independent set solvers.

    \item This provide us with lower estimate $m_1(\bbbr^2) \ge \frac{|M|}{n \cdot m}$.
    
\end{enumerate}

Let $V_{n, m} = \{v_{0, 0}, v_{1, 0}, \dots, v_{n, 0}, v_{0, 1}, \dots, v_{n-1, m-1} \} \subset T_{l_1, l_2, \alpha}$, where $v_{i, j} = \left( \frac{i}{n}, \frac{j}{m} \right) \in [0, 1)^2$, be the vertices (parametrized by their affine coordinates) of the graph $G_{n, m} = \left(V_{n, m}, E_{n, m} \right)$ constructed on the flat torus $T_{l_1, l_2, \alpha}$.

The following lemma clarifies the $d$-regularity of the constructed graph based on its invariance under shifts by vectors $\frac{1}{n}\vec{v}_1$ and $\frac{1}{m}\vec{v}_2$, respectively. This property also accelerates the construction of $G_{n, m}$ in the proposed experimental pipeline.

\begin{lemma}
Let $V_{0, 0} = \left \{ v_{s_1, t_1}, v_{s_2, t_2}, \dots, v_{s_k, t_k}  \right \} \subset V_{n, m}$ be the set of all neighbors of vertex $v_{0, 0}$ in the graph $G_{n, m}$. Then, for all $i \in \{0, 1, \dots, n - 1 \}$ and $j \in \{0, 1, \dots, m - 1 \}$ the set $V_{i, j} = \left \{ v_{(s + i) \bmod n, (t + j) \bmod m} : \; \forall \; v_{s, t} \in V_{0, 0} \right \}$ is the set of all neighbors of vertex $v_{i, j}$ in the graph $G_{n, m}$.

\label{lemma_fast_graph_construction}
\end{lemma}

\begin{proof}

First, note that for every $i, i_1 \in \{0,\dots, n - 1\}$ and $j, j_1 \in \{0, \dots, m - 1\}$ the following equation holds: 

\begin{align*}
\rho(v_{i, j}, v_{i_1, j_1}) &= \rho \left( \left(\frac{i}{n}, \frac{j}{m} \right), \left(\frac{i_1}{n}, \frac{j_1}{m} \right) \right) = \\
&= \min \left\{ \left| \left(\frac{i_1 - i}{n} + k_1\right) \cdot \vec{v}_1 + \left( \frac{j_1 - j}{m} + k_2 \right) \cdot \vec{v}_2 \right | : k_1, k_2 \in \bbbz \right\} = \\ 
&= \min \left\{ \left| \left( \frac{s}{n} + k_1 \right) \cdot \vec{v}_1 + \left( \frac{t}{m}  + k_2 \right) \cdot \vec{v}_2 \right | : k_1, k_2 \in \bbbz \right\} = \\ 
&= \rho \left( \left(\frac{s}{n}, \frac{t}{m} \right), \left(0, 0 \right) \right) = \rho \left(v_{s, t}, v_{0, 0} \right),
\end{align*}

where $s = (i_1 - i + n) \bmod n$ and $t = (j_1 - j + m) \bmod m$. Here and after, ``$\bmod$'' denotes the remainder of the modulo division. Using this fact and the edge construction rules for $G_{n, m}$, we obtain the following:

\begin{align*}
    (v_{i, j}, v_{i_1, j_1}) & \in E_{n,m} \Leftrightarrow \rho(v_{i, j}, v_{i_1, j_1}) \in [1 - 2r, 1 + 2r] \Leftrightarrow \\
    & \Leftrightarrow \rho(v_{s, t}, v_{0, 0}) \in [1 - 2r, 1 + 2r] \Leftrightarrow (v_{s, t}, v_{0, 0}) \in E_{n,m} \Leftrightarrow v_{s, t} \in V_{0, 0}. 
\end{align*}

Therefore, we prove that $V_{i_1, j_1}$ is a neighbor of $v_{i, j}$ if and only if $v_{s, t}$ is a neighbor of $v_{0, 0}$. From the equalities $i_1 = (s + i) \bmod n$ and $j_1 = (t + j) \bmod m$,  $$V_{i, j} = \left \{ v_{i_1, j_1} : \; (v_{i, j}, v_{i_1, j_1}) \in E_{n,m} \right \} = \left \{ v_{(s + i) \bmod n, (t + j) \bmod m} : \; \forall \; v_{s, t} \in V_{0, 0} \right \}.$$
\qed
\end{proof}

According to Lemma~\ref{lemma_fast_graph_construction}, we first compute the distances between vertices $v_{0, 0}$ and $v_{s, t}$ for every $s \in \{0, 1, \dots, n - 1\}$ and $t \in \{0, 1, \dots, m - 1\}$, next calculate the distances $\rho(v_{s, t}, v_{0, 0})$ according to Lemma~\ref{metric_computation}. This allows us to construct the set $V_{0, 0}$ of neighbors of $v_{0, 0}$. For each other vertex $v_{i, j}$ in this graph, we construct the set of its neighbors $V_{i, j}$ using Lemma~\ref{lemma_fast_graph_construction} without additional distance computations. Thus, the key importance of the Lemma~\ref{lemma_fast_graph_construction} is that it allows us to construct the graph $G_{n, m}$ with only $n \times m$ times (linear in the number of vertices) of $\rho$ calculations. Additionally, this lemma states that all vertices in the considered graph $G_{n, m}$ have equal degrees, making $G_{n, m}$ a $d$-regular graph for every $n, m \in \bbbn$.

According to Theorem~\ref{main_theorem} and the experiment scheme, a key aspect of estimating lower bound for $m_1(\bbbr^2)$ is finding the Maximum Independent Set (MIS) in the graph $G_{n, m}$ corresponding to any perfectly periodic flat torus. Various approaches for the MIS problem solving are considered in the next subsection.

\subsection{Independent set solvers and machine learning for combinatorial optimization}

In this part of the paper, the Maximum Independent Set (MIS) problem is formally introduced, the solvers used in our analysis are described, and we discuss relevant research in the broader field of deep learning for combinatorial optimization. 

Given an undirected graph $G = (V, E)$, an independent set is a subset of vertices $S \subset V$ where no two vertices in $S$ are connected by an edge. Formally, for all vertices $u, v \in S$, $(u, v) \notin E$. Let each vertex $u \in V$ have an associated weight $w_u$. $IS(G)$ represents the set of all independent sets of $G$. The Maximum Weighted Independent Set (MWIS) problem seeks to find the independent set $S$ that maximizes the sum of its vertex weights:
$$\displaystyle\arg\max_{S \in IS(G)}\sum\limits_{u \in S}w_u.$$

The unweighted MIS problem is a special case of MWIS where all vertices have equal weights, i.e., $w_u = 1$ for all $u \in V$. Both problems are strongly NP-complete \cite{Garey1978S}. According to the aim of this paper, we focus on solving the MIS problem.

For the experiments in this paper the open-source independent set benchmarking suite MIS-Benchmark \cite{mis_benchmark2022} for the NP-hard Maximum Independent Set search was used. This suite integrates several solvers into a single, easily accessible command-line interface using Anaconda \cite{Anaconda2020} and provides a unified input and output format. The framework ensures the proper invocation of the solvers and enables running experiments on various solvers in different configurations.

We consider several solvers (KaMIS, Intel-TreeSearch, DGL-TreeSearch, \\ Learning What to Defer) integrated into MIS-Benchmark in more details. Note that despite good maximal independent set size approximations, none of these methods guarantee that the found set will be a true maximal independent set for a given graph.

\textbf{KaMIS.} KaMIS is an open-source solver specifically designed for the MIS and MWIS problems. It provides support for both unweighted cases \cite{Lamm2017,Hespe2019} and weighted cases \cite{Lamm2019_weight}. The solver utilizes a graph kernelization and an optimized branch-and-bound algorithm to effectively find independent sets. It is important to note that the algorithms and techniques employed differ between the weighted and unweighted scenarios. The MIS-Benchmark \cite{mis_benchmark2022} uses the code unmodified from the official KaMIS repository \cite{kamis_repo}.

\textbf{Intel-TreeSearch.} In the influential paper \cite{Li2018_combopt}, Zhuwen Li et al. (2018) propose a guided tree search algorithm to find maximum independent sets of a graph. Their key idea is to train a graph convolutional network (GCN) \cite{KipfWelling2017_gcn}, which assigns each vertex a probability of belonging to the independent set. The algorithm then greedily and iteratively assigns vertices to the set based on these probabilities. Furthermore, they employ the reduction and local search algorithms from KaMIS to speed up computations. The authors of MIS-Benchmark \cite{mis_benchmark2022} modified some code from the original Intel-TreeSearch implementation \cite{intel_repo} and added command-line flags for more fine-grained control of the solver configuration.

\textbf{DGL-TreeSearch.} The code provided in \cite{Li2018_combopt} may be difficult to read and maintain, which increases the likelihood of errors in evaluation. To address this, the authors of MIS-Benchmark \cite{mis_benchmark2022} re-implemented the tree search using PyTorch \cite{Pytorch2019} and the well-established Deep Graph Library (DGL) \cite{wang2019dgl}. Their implementation aims to provide a more readable and modern solution that benefits from recent improvements in these deep learning libraries. Additionally, they addressed various issues with the original implementation that sometimes deviates from the paper. Notably, they also implemented additional techniques to enhance the search process, such as queue pruning and weighted selection of the next element, as well as multi-GPU functionality.

\textbf{Learning What to Defer.} We also examine Learning What to Defer (LwD), an unsupervised deep reinforcement learning-based solution proposed in \cite{LwD2020}. The concept behind LwD is analogous to tree search, where the algorithm iteratively assigns vertices to the independent set. However, instead of employing a supervised graph convolutional network, it utilizes an unsupervised agent built upon the GraphSAGE architecture \cite{GraphSAGE2017} and trained using Proximal Policy Optimization \cite{Schulman2017ProximalPO}. Unlike other considered approaches, there is no queue of partial solutions. For further details on the algorithm, we refer readers to the original paper and repository \cite{lwd_repo}.

To conduct the experiments in this paper, the code of MIS-Benchmark \cite{mis_benchmark2022} was slightly modified by adding new command-line flags and fixing some typos.

\subsection{Evaluation and comparison of various solvers on flat torus based graphs}

To evaluate and compare the approaches from the previous subsection, we provide experiments with graph datasets. Each graph $G_{n, m}$, constructed from the flat torus $T_{l_1, l_2, \alpha}$, is parameterized via 5 parameters: $l_1, l_2$, and $\alpha$ (flat torus parameters) and grid sizes $n, m$.

Firstly, the Dataset-1 constructed in the following way. We iterate over length values $l_1, l_2 \in [2, 6]$ with a step size of $0.2$ and angle values $\alpha \in [\frac{\pi}{9}; \frac{\pi}{2}]$ with a step size of $\frac{\pi}{36}$. For each combination of $(l_1, l_2, \alpha)$, we check if the flat torus $T_{l_1, l_2, \alpha}$ is perfectly periodic and if $l_1 \le l_2$. If both conditions are met, we add the set of parameters $(l_1, l_2, \alpha)$ to the dataset. This ensures that all objects in the dataset are unique (due to $l_1 \le l_2$) and that each torus is perfectly periodic, allowing for estimation of $m_1(\bbbr^2)$. The size of Dataset-1 is 2986 (see Table~\ref{tab1}). The choice of uniform grid bounds for the graph parameters is conditioned on the following:  $l_{min} \cdot \sin \alpha_{max} = 2 \cdot \sin{\pi/2} = 2$ and $l_{max} \cdot \sin \alpha_{min} = 6 \cdot \sin{\pi/9} \approx 2.05$. This guarantees that all length and angle values (from defined uniform grids) are represented in the dataset as one of graph parameters. Finally, we iterate the experimental scheme over all graphs $G_{n, m}$ (where $n = m = 100$) constructed using the parameters stored in the dataset.

We compare various solvers from the MIS-Benchmark on the constructed Dataset-1. A solver time limit of 100 seconds was set for all approaches, or the machine learning-based approaches (Intel-TreeSearch, DGL-TreeSearch, and LwD), pretrained checkpoints from the MIS-Benchmark repository \cite{mis_benchmark2022} were used.

The results of this comparison are presented in Table~\ref{tab2}. Each column contains the mean independent set size found via the corresponding method over all graphs in a given dataset. Note that all graphs in each considered dataset have the same number of vertices, making the mean independent set size over the dataset an appropriate metric for comparing different approaches to the MIS finding problem.

Next, we identify a local maximum in the mean independent set size found by the four analyzed approaches: KaMIS, DGL-TreeSearch, Intel-TreeSearch, and LwD. Dataset-2 is constructed around this local maximum point on Dataset-1 by decreasing the grid steps for $l_1$, $l_2$, and $\alpha$ values. This dataset is constructed using truncated uniform grids for sides $l_1$, $l_2$ and the angle $\alpha$. We then add to Dataset-2 sets of parameters $(l_1, l_2, \alpha)$ where the flat torus $T_{l_1, l_2, \alpha}$ is perfectly periodic and $l_1 \le l_2$. The grid sizes used in this experiment ($n = m = 100$), corresponding to the number of graph vertices, are the same for all datasets. Subsequent datasets are constructed analogously around the local maximum point $(l_1^*, l_2^*, \alpha^*)$ over the previous dataset.

Table~\ref{tab1} illustrates the construction of the datasets used in this paper. We highlight that steps for $l_1, l_2, \alpha$ were decreased and the grid for each following dataset is more frequent than for the previous one.

\begin{table}

\caption{Datasets constructions that used in the experiments.}\label{tab1}
\begin{tabular}{|c|c|c|c|c|c|c|}
\hline
Dataset &  $l_i$ range & $l_i$ step  & $\alpha$ range  & $\alpha$ step & dataset size & $(l_1^*, l_2^*, \alpha^*)$\\
\hline
Dataset-1 & $[ 2, 6]$ &  0.2 &   $\left[20^\circ, 90^\circ\right]$  & $5^\circ$ & 2986 & $(3.4, 3.4, 60^\circ)$ \\
\hline
Dataset-2 & $[ 3.2, 3.6]$ &  0.02 &  $\left[55^\circ, 65^\circ\right]$  & $2.5^\circ$ & 1155 & $(3.34, 3.34, 60^\circ)$ \\
\hline
Dataset-3 & $[ 3.32, 3.36]$ &  0.004  &  $\left[57.5^\circ, 62.5^\circ\right]$  & $0.5^\circ$ & 726 & $(3.332, 3.336, 60^\circ)$ \\
\hline
Dataset-4 & $[ 3.328, 3.340]$ &  0.001  &  $\left[59.5^\circ, 60.5^\circ\right]$  & $0.25^\circ$ & 455 & $(3.331, 3.331, 60^\circ)$ \\
\hline
\end{tabular}

\end{table}

Table~\ref{tab2} shows the comparison of different methods over various datasets. The bold font indicates the best approach (out of four), the italic font corresponds to the second best method. In could be noticed that KaMIS approach shows best results over all datasets, however, the second place can be taken by any of the other three neural network based approaches. The value in brackets in the Table~\ref{tab2} indicates the relative deviation of the mean independent set size from the best method. Experiment details and full comparison results, including CSV files, are contained in our repository \cite{our_repo}.

\begin{table}
\begin{center}
\caption{Comparison of various MIS finding approaches for the flat tours datasets: average MIS size over graphs with 10000 vertices ($n = m = 100$).}\label{tab2}
\begin{tabular}{|c|c|c|c|c|}
\hline
Dataset name &  KaMIS & DGL-TreeSearch & Intel-TreeSearch & LwD \\
\hline
Dataset-1 & \textbf{1781.9 (1.00)} &  1744.3 (0.98) &  \textit{1757.8 (0.99)} & 1648.3 (0.92) \\
\hline
Dataset-2 & \textbf{2054.2 (1.00)} &  \textit{2026.6 (0.99)} &  1994.0 (0.97) & 1990.3 (0.97)  \\
\hline
Dataset-3 & \textbf{2133.7 (1.00)} &  2095.9 (0.98) &  2069.8 (0.97) & \textit{2098.5 (0.98)} \\
\hline
Dataset-4 & \textbf{2142.2 (1.00)} &  2112.9 (0.99) &  2060.96 (0.96) & \textit{2115.1 (0.99)} \\
\hline
\end{tabular}
\end{center}
\end{table}

This experiment allows us to find the local maximum point step-by-step by decreasing the grid steps on the next dataset. As a result, the local maximum over Dataset-4 is determined to be: $l_1^* = 3.331$, $l_2^* = 3.331$, $\alpha^* = 60^\circ = \pi/3$. Let's visualize on the plot the independent set found by the four considered approaches, depending on the number of vertices in the graph $G_{n, m}$ constructed from the flat torus $T_{l_1^*, l_2^*, \alpha^*}$ with optimal parameters from the last dataset.

After determining the local maximum parameters, we conducted experiments by varying the number of vertices in graphs $G_{n, m}$ obtained via the flat torus $T_{l_1^*, l_2^*, \alpha^*}$. Next, we found the maximal independent set $M$ (using one of the four considered solvers) with size $|M|$. This allowed us to obtain a lower bound estimate of the $m_1(\bbbr^2)$ value as $m_1(\bbbr^2) \ge \frac{|M|}{nm}$. Without loss of generality, we have considered the case $n = m$, so the graph $G_{n, n}$ with $n^2$ vertices was created and the corresponding lower bound was calculated via $m_1(\bbbr^2) \ge \frac{|M|}{n^2}$.

Following figures show the maximal independent sets obtained using various approaches for solving the MIS problem on graphs $G_{n, n}$ constructed via the flat torus $T_{l_1^*, l_2^*, \alpha^*}$ with varying numbers of vertices: KaMIS (Fig.~\ref{kamis_exp_results}), DGL-TreeSearch (Fig.~\ref{dgl_exp_results}), Intel-TreeSearch (Fig.~\ref{intel_exp_results}), LwD (Fig.~\ref{lwd_exp_results}). The experiment settings were the same as those used in the previous experiments with datasets mentioned in Table~\ref{tab1}. More details can be found in the repository related to this paper \cite{our_repo}.

\begin{figure}
\begin{center}
\begin{tabular}{ccc}
    {\includegraphics[scale=0.085]{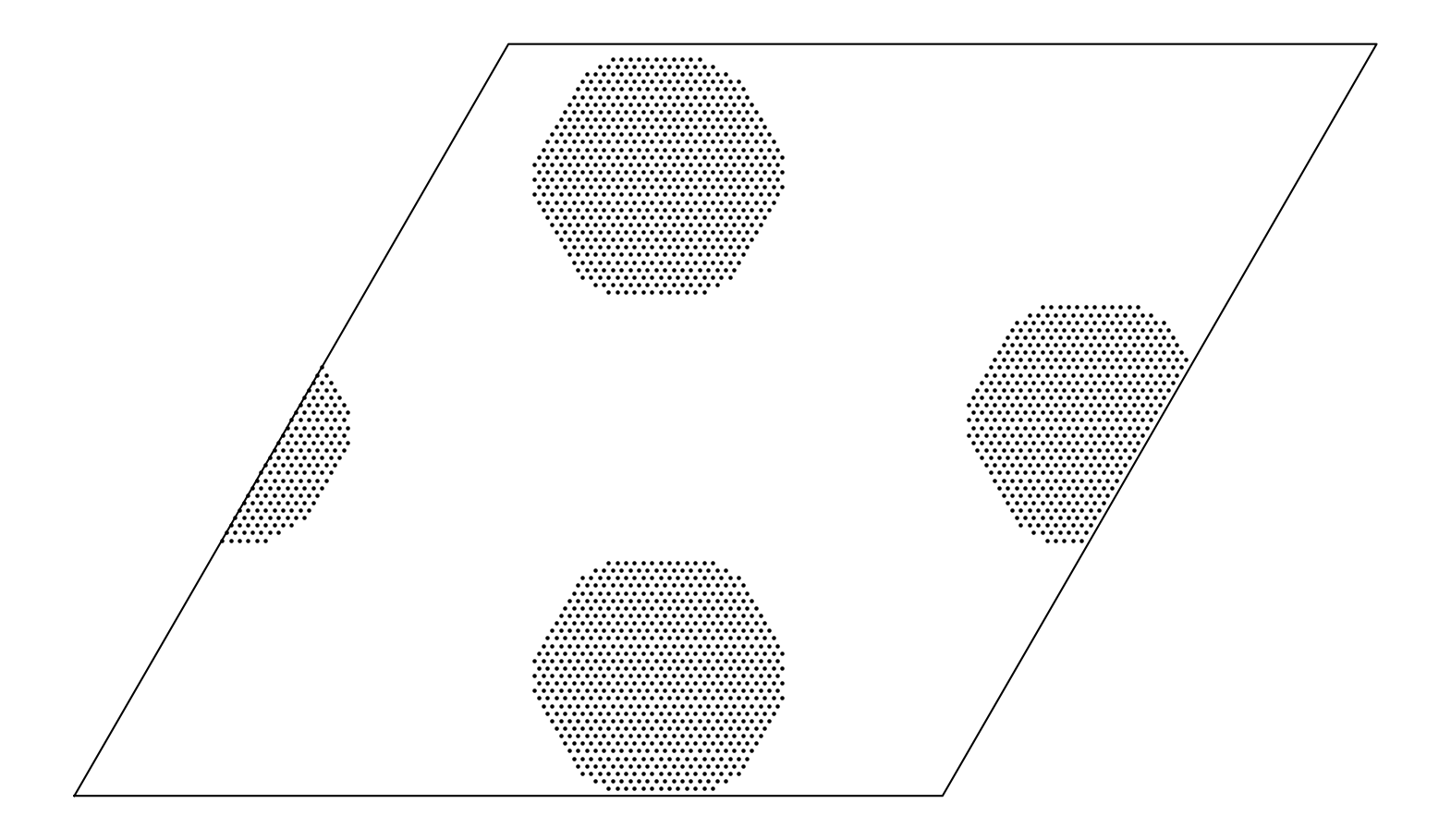}} &
    {\includegraphics[scale=0.085]{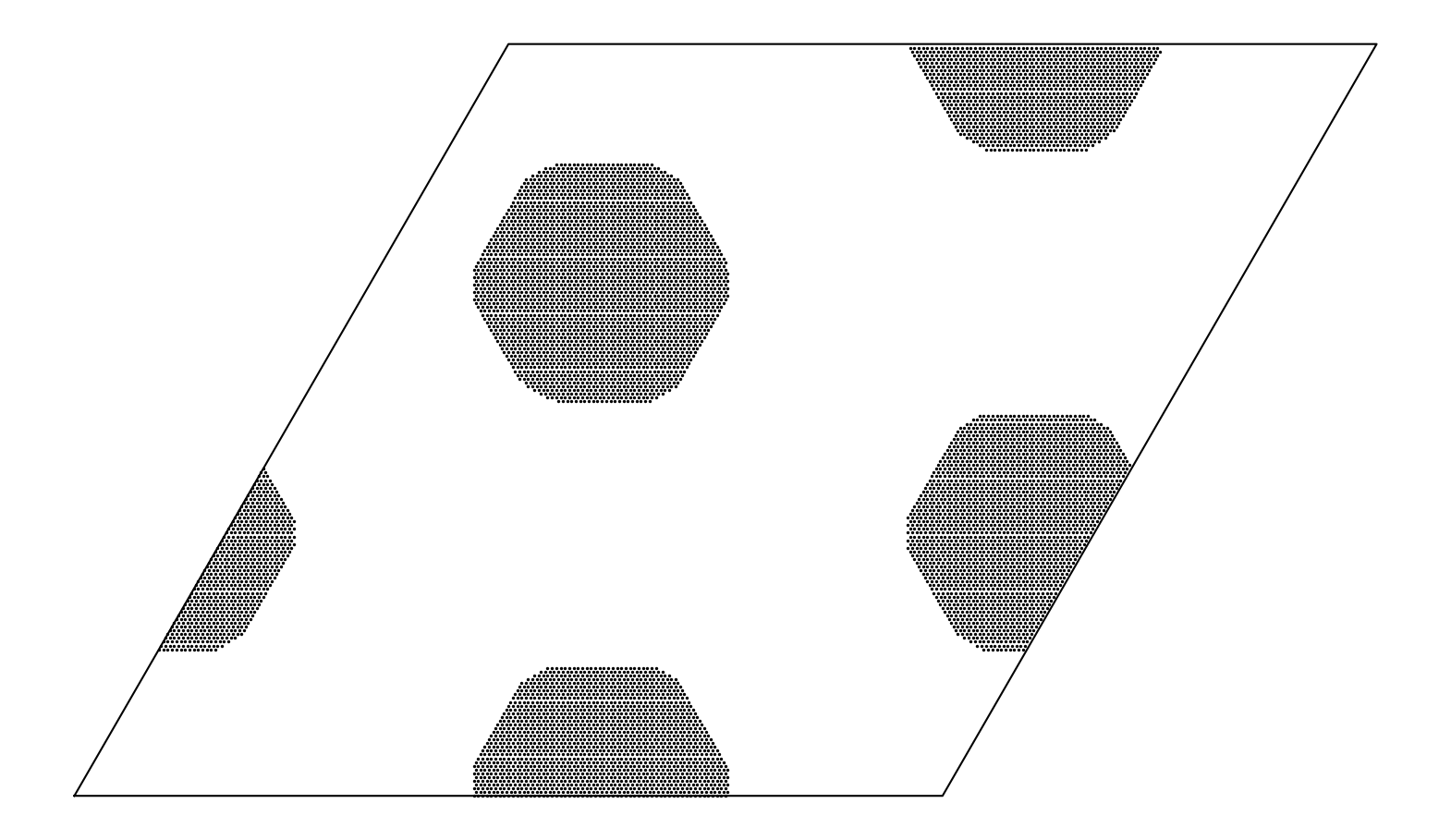}} &
    {\includegraphics[scale=0.085]{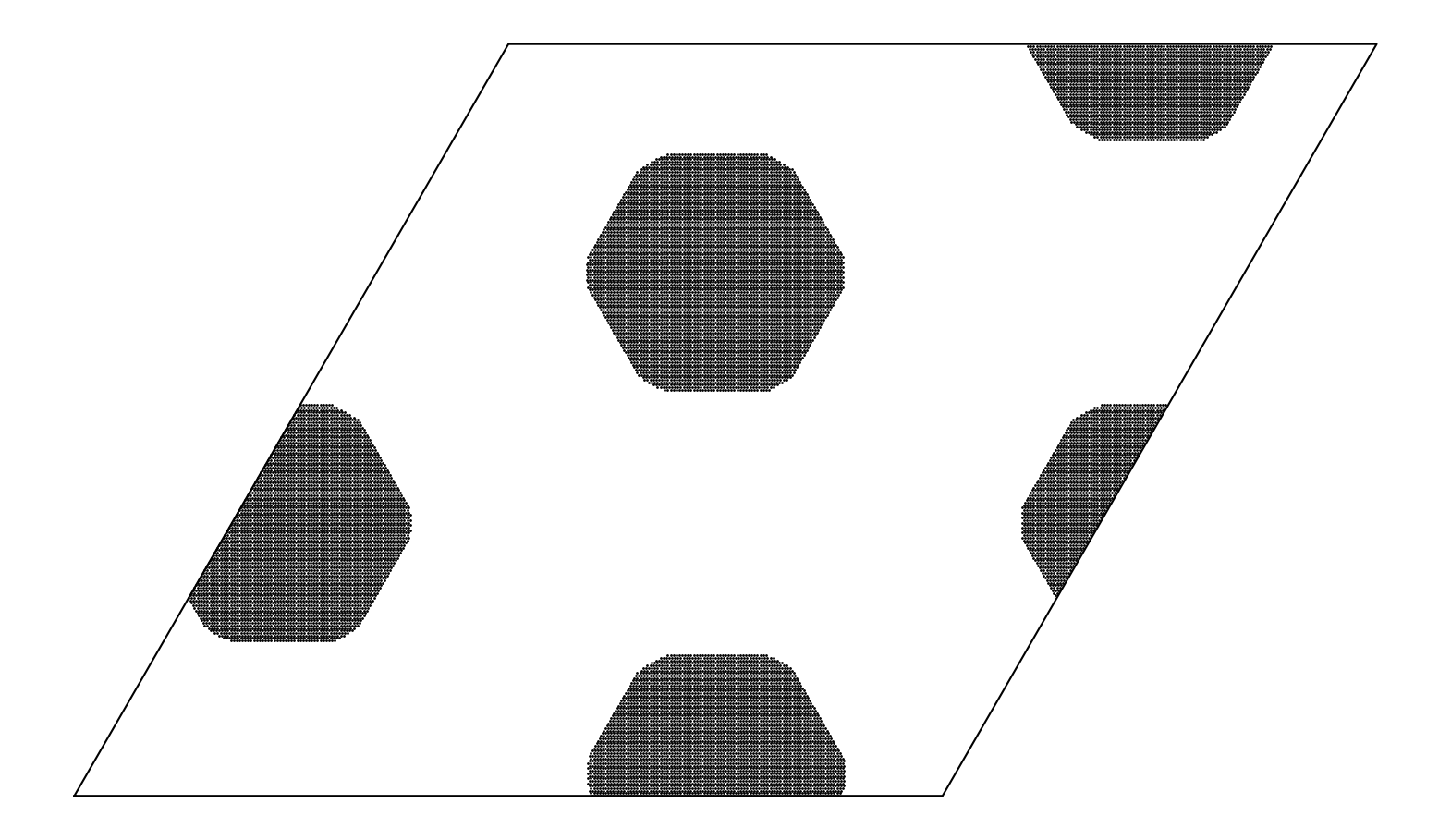}}\\
    $n = 100, |M| = 2193,$  & $n = 200, |M| = 8850,$ & $n = 300, |M| = 19965,$ \\
    $m_1(\bbbr^2) \ge 0.2193$ & $m_1(\bbbr^2) \ge 0.2212$  & $m_1(\bbbr^2) \ge 0.2218$
\end{tabular}
\caption{KaMIS results for $G_{n, n}$ graphs based on torus $T_{l_1^*, l_2^*, \alpha^*}$} \label{kamis_exp_results}
\end{center}
\end{figure}
\begin{figure}
\begin{center}
\begin{tabular}{ccc}
    {\includegraphics[scale=0.085]{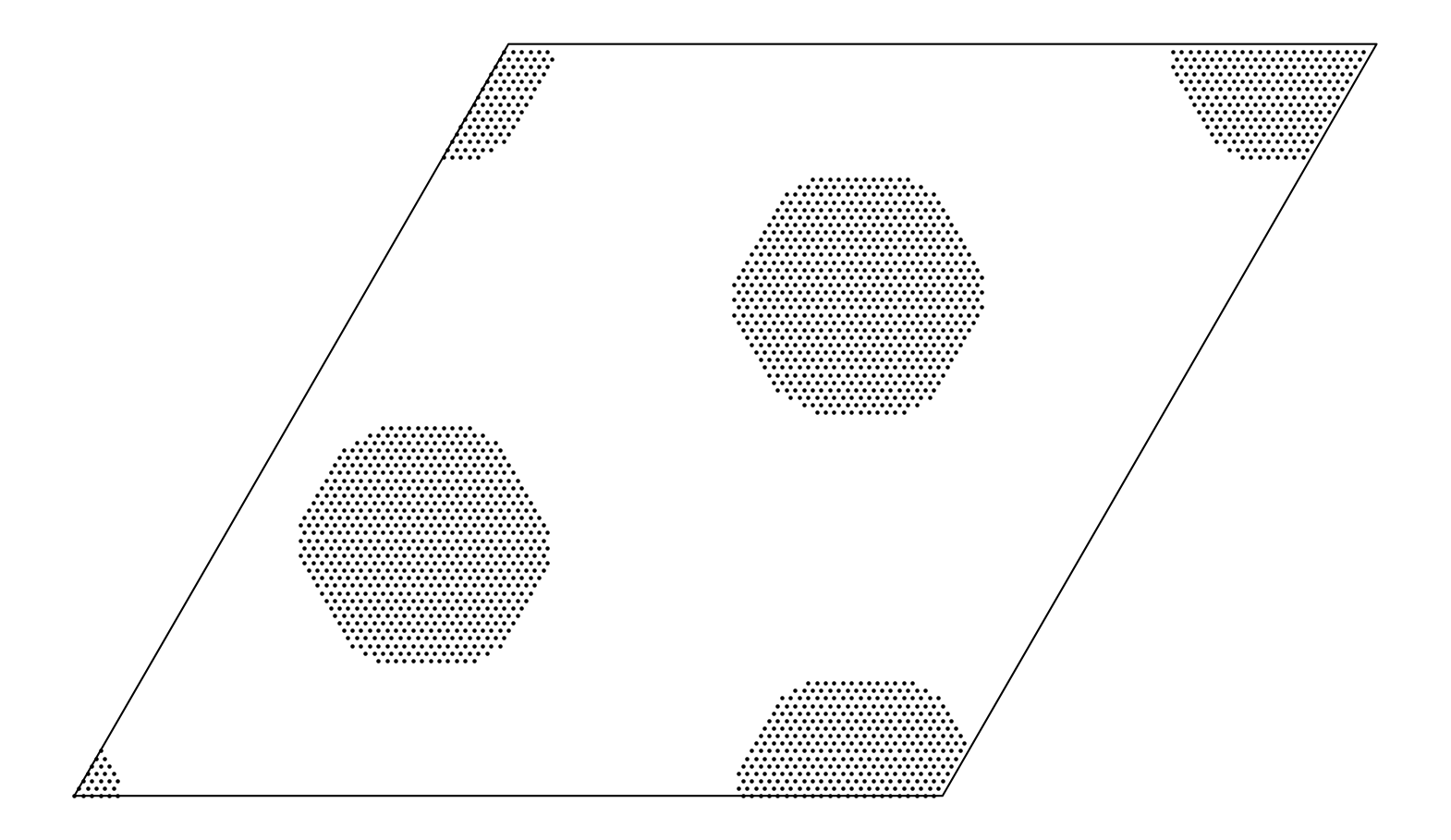}} &
    {\includegraphics[scale=0.085]{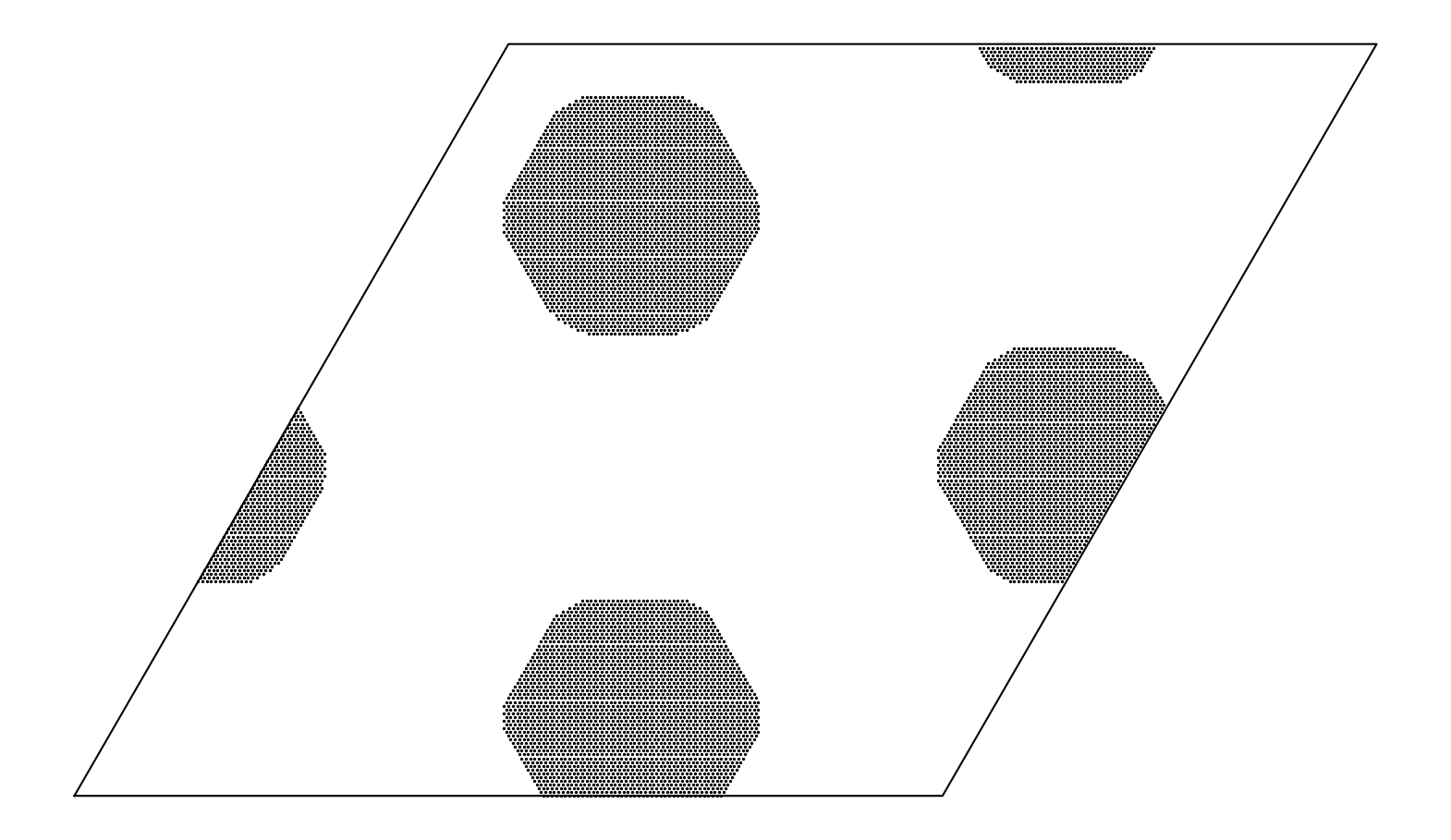}} &
    {\includegraphics[scale=0.085]{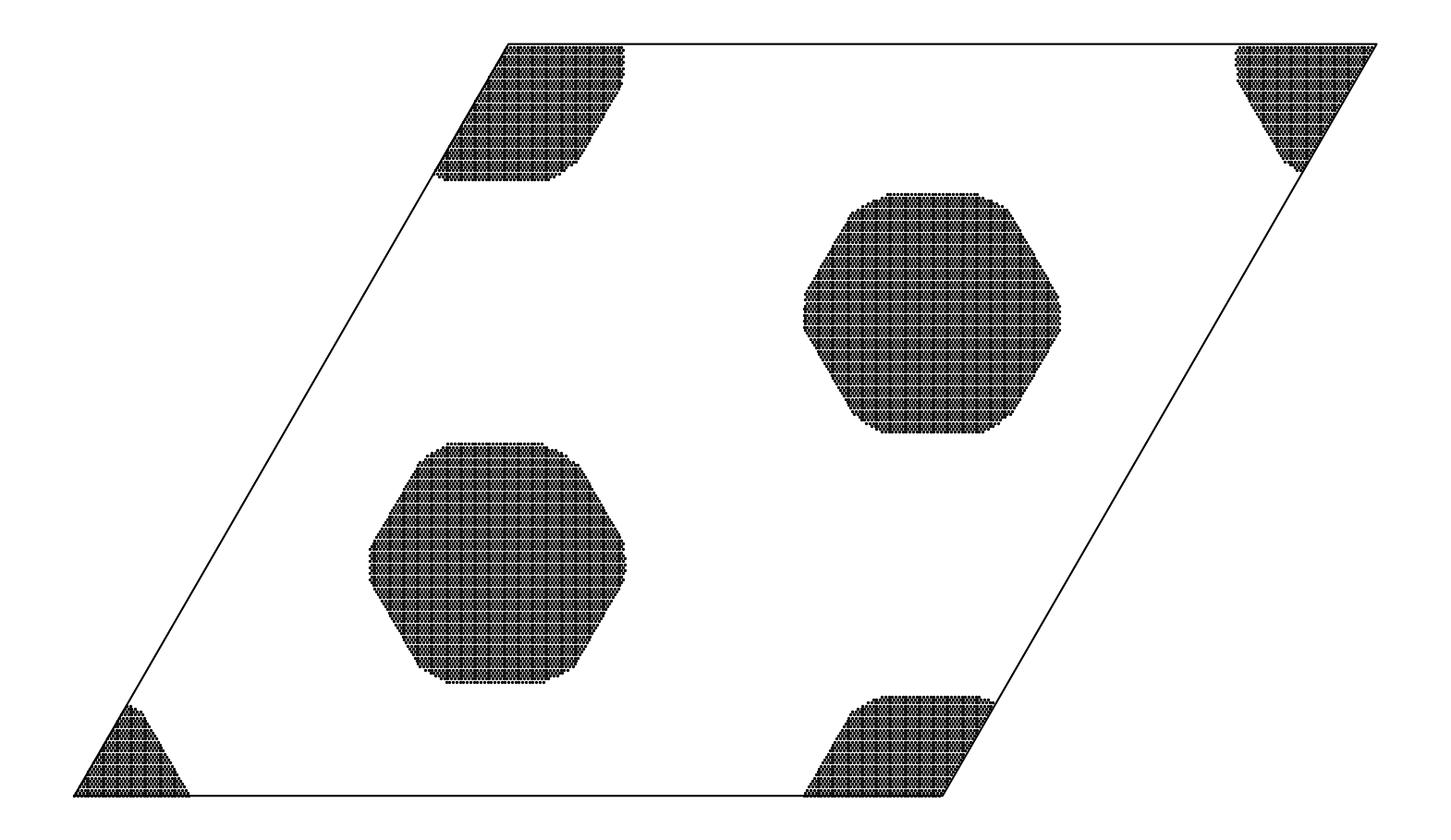}} \\
    $n = 100, |M| = 2193,$  & $n = 200, |M| = 8850,$ & $n = 250, |M| = 13962,$ \\
    $m_1(\bbbr^2) \ge 0.2193$ & $m_1(\bbbr^2) \ge 0.2212$  & $m_1(\bbbr^2) \ge 0.2233$
\end{tabular}
\caption{DGL-TreeSearch results for $G_{n, n}$ graphs based on torus $T_{l_1^*, l_2^*, \alpha^*}$} \label{dgl_exp_results}
\end{center}
\end{figure}
\begin{figure}
\begin{center}
\begin{tabular}{ccc}
    {\includegraphics[scale=0.085]{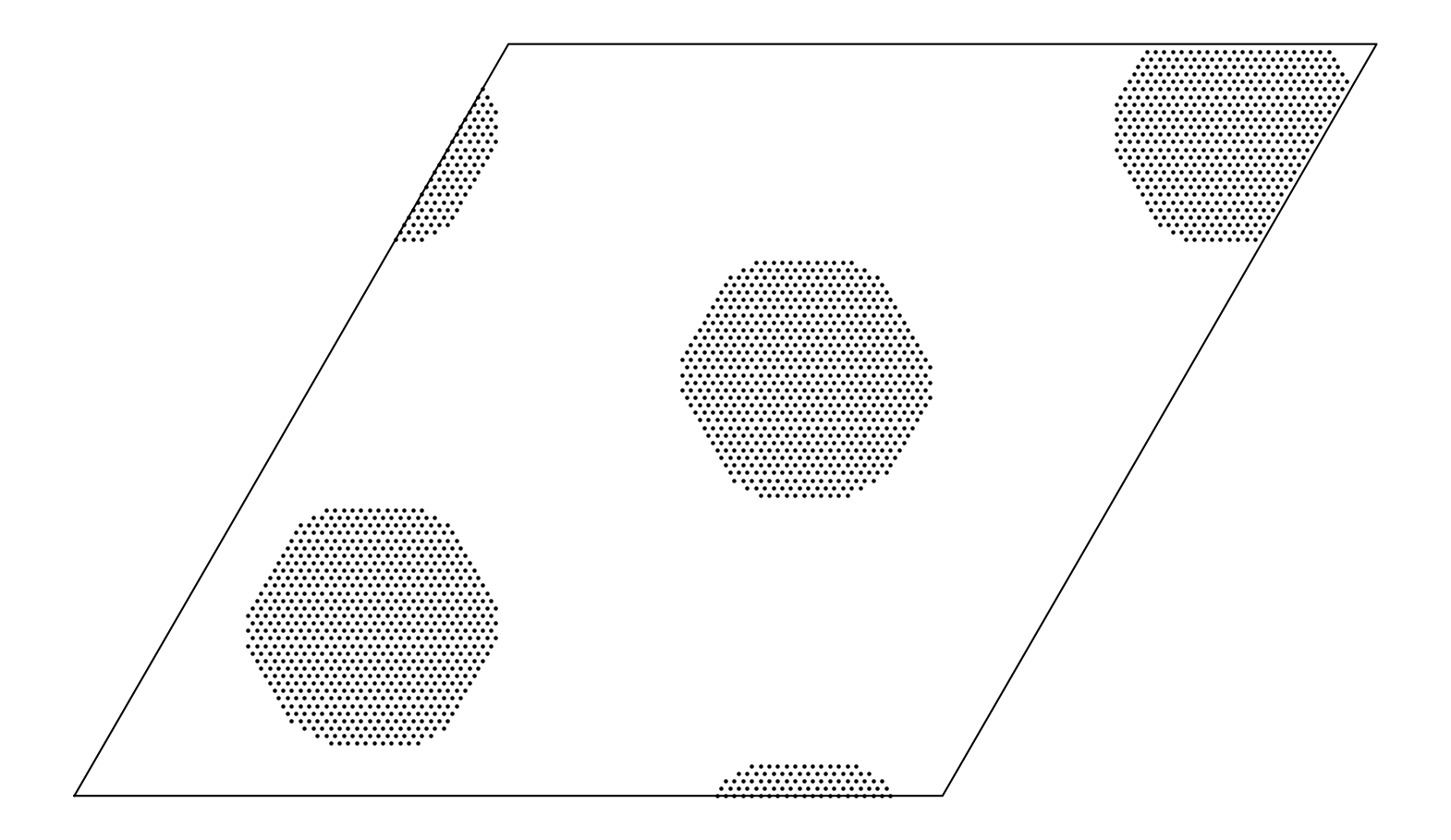}} &
    {\includegraphics[scale=0.085]{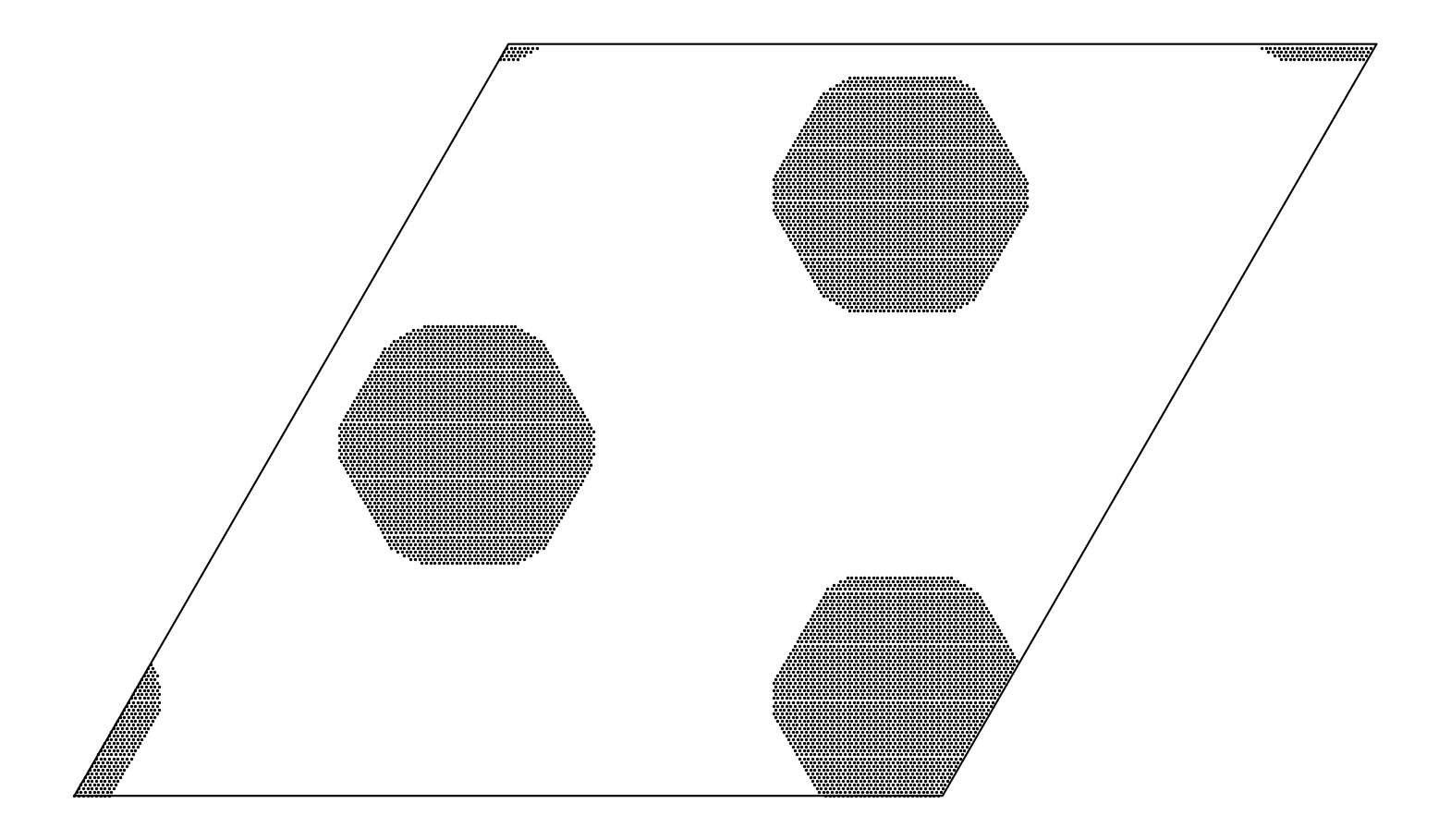}} &
    {\includegraphics[scale=0.085]{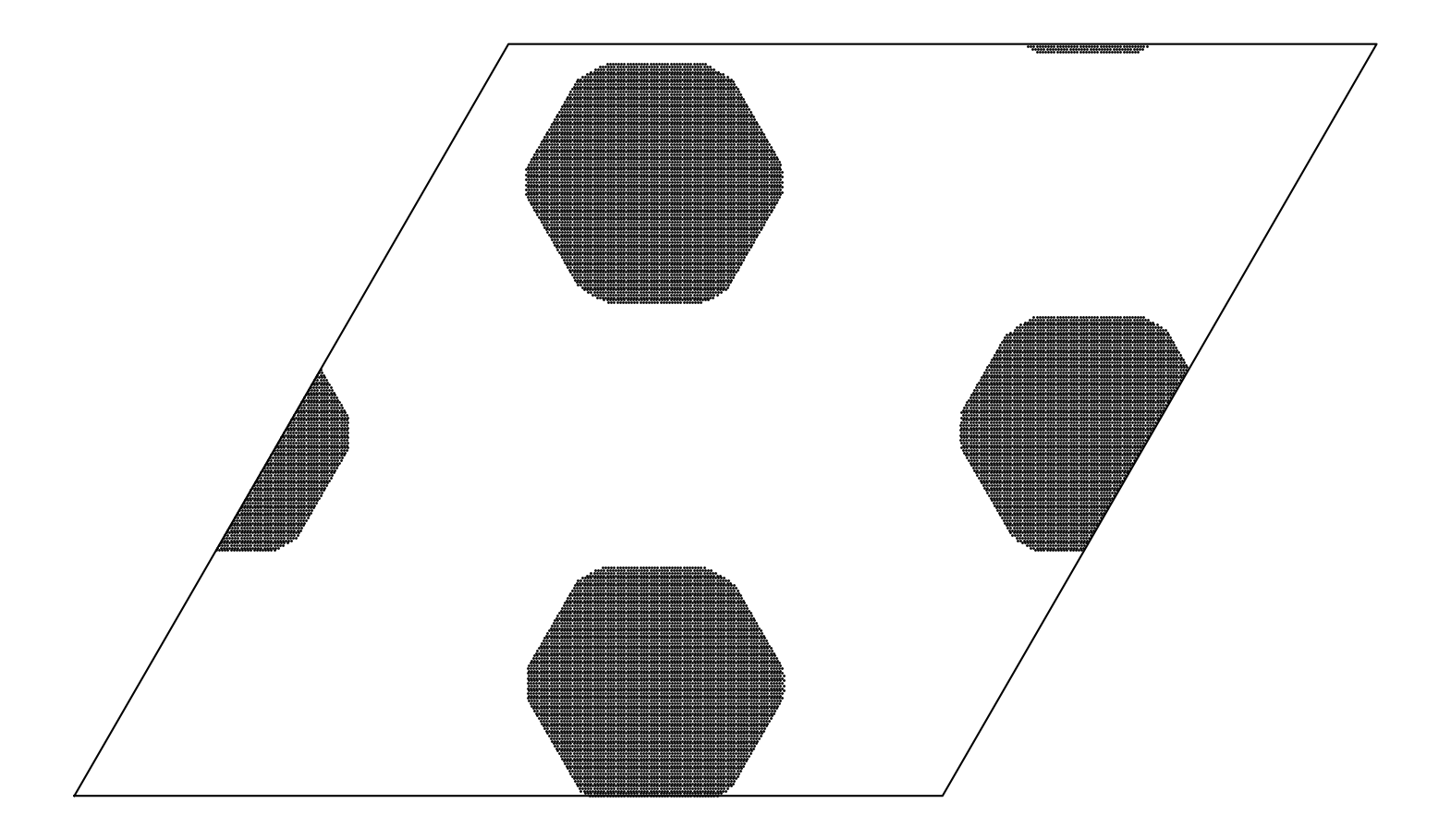}} \\
    $n = 100, |M| = 2193,$  & $n = 200, |M| = 8850,$ & $n = 300, |M| = 19960,$ \\
    $m_1(\bbbr^2) \ge 0.2193$ & $m_1(\bbbr^2) \ge 0.2212$  & $m_1(\bbbr^2) \ge 0.2217$
\end{tabular}
\caption{Intel-TreeSearch results for $G_{n, n}$ graphs based on torus $T_{l_1^*, l_2^*, \alpha^*}$} \label{intel_exp_results}
\end{center}
\end{figure}
\begin{figure}
\begin{center}
\begin{tabular}{ccc}
    {\includegraphics[scale=0.085]{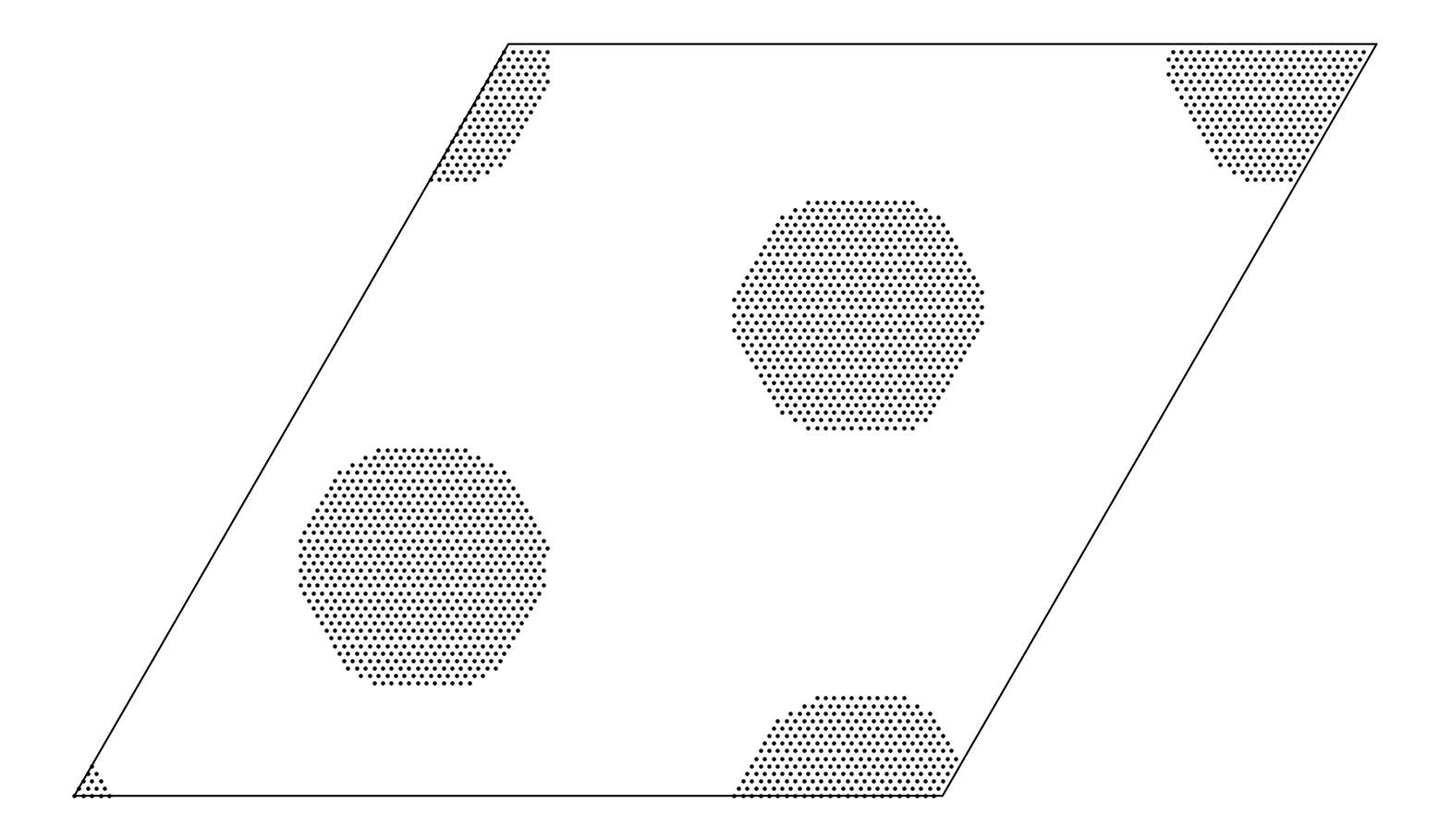}} &
    {\includegraphics[scale=0.085]{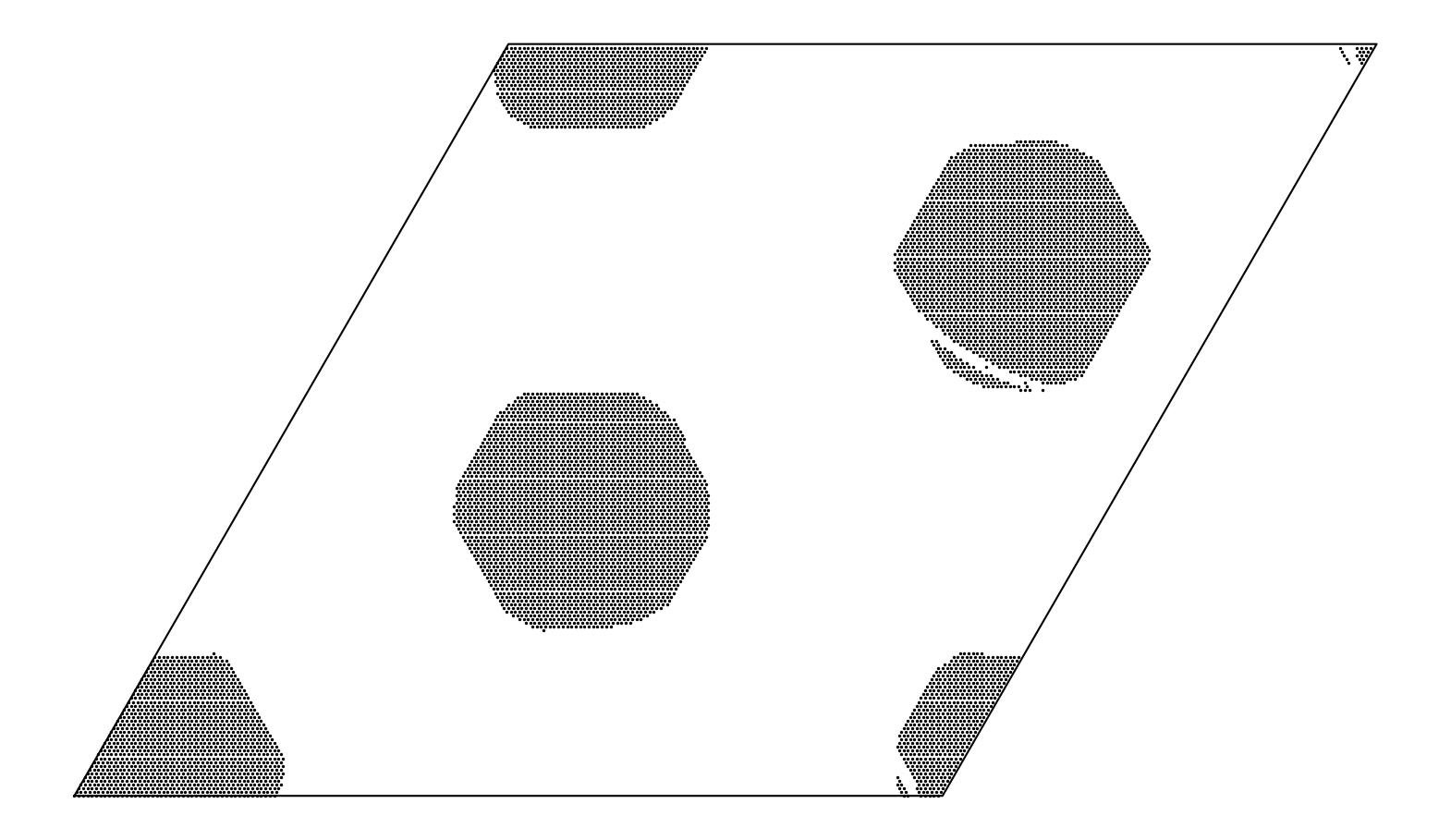}} &
    {\includegraphics[scale=0.085]{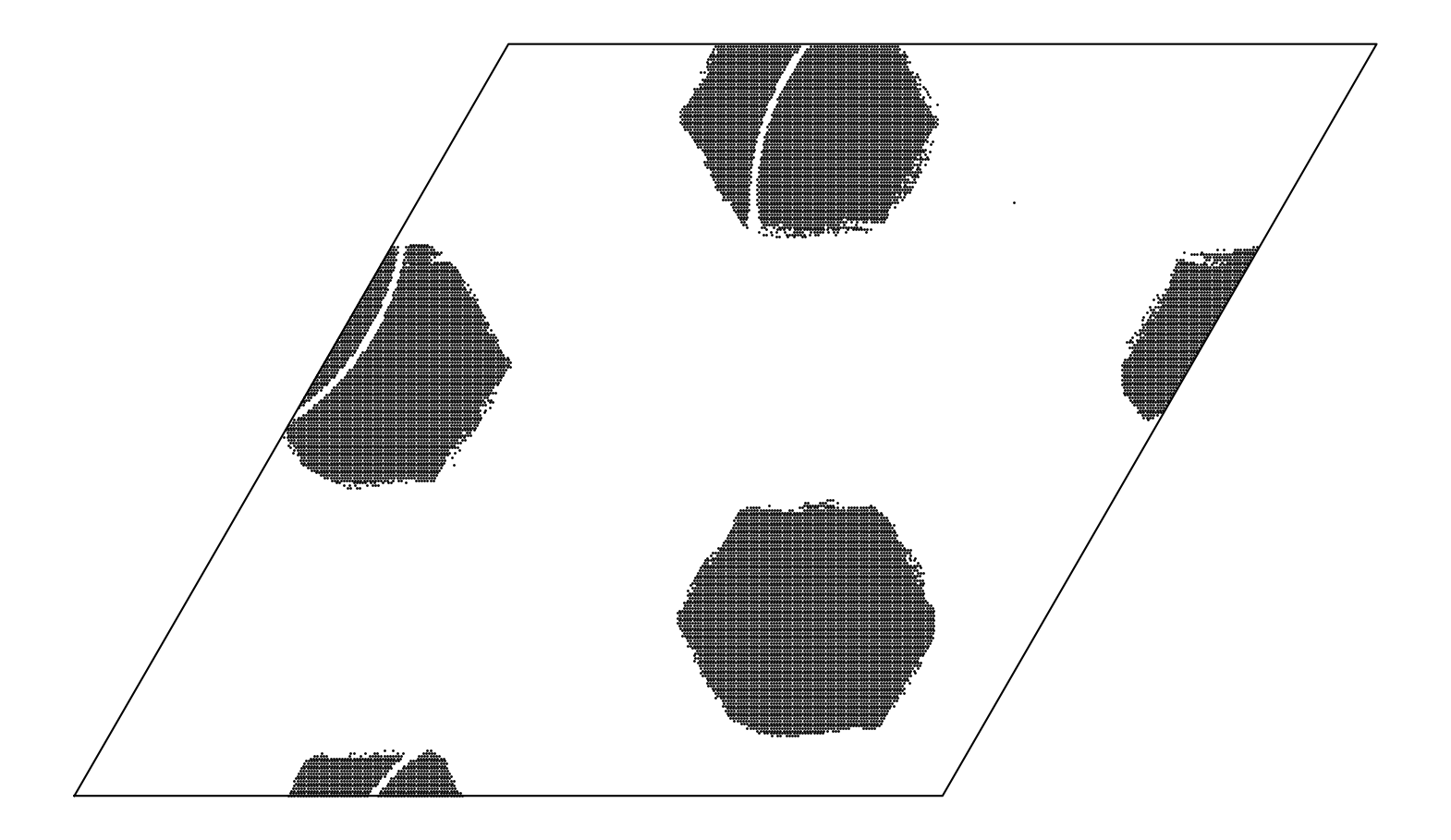}} \\
    $n = 100, |M| = 2190,$  & $n = 200, |M| = 8633,$ & $n = 300, |M| = 18329,$ \\
    $m_1(\bbbr^2) \ge 0.2190$ & $m_1(\bbbr^2) \ge 0.2158$  & $m_1(\bbbr^2) \ge 0.2036$
\end{tabular}
\caption{LwD results for $G_{n, n}$ graphs based on torus $T_{l_1^*, l_2^*, \alpha^*}$} \label{lwd_exp_results}
\end{center}
\end{figure}

\subsection{Analysis of experimental results}

This section presents an analysis of the results obtained from experiments conducted with graphs constructed from the flat torus $T_{l_1^*, l_2^*, \alpha^*}$.

While KaMIS, DGL-TreeSearch, and Intel-TreeSearch yielded similar or nearly identical independent set sizes for $n = 100$ and $n = 200$, DGL-TreeSearch outperformed KaMIS and Intel-TreeSearch for larger values of $n$. Notably, the DGL-based $m_1(\bbbr^2)$ estimation for $n = 250$ surpassed the lower bounds obtained using KaMIS and Intel-TreeSearch, even for $n = 300$. However, DGL-TreeSearch, as described in \cite{Li2018_combopt}, maintains a queue of solutions. Therefore, despite employing queue pruning techniques, the DGL approach necessitates greater RAM memory compared to the other three approaches investigated in this study. Consequently, this approach exhibits difficulty scaling to graphs with a large number of vertices. Our experiments involved graphs with up to $250^2 = 62,500$ vertices for DGL and up to $300^2 = 90,000$ vertices for the remaining three approaches.

Another notable observation pertains to the reinforcement learning-based approach Learning What to Defer (LwD) \cite{LwD2020}. When applied to large graphs, LwD consistently yielded smaller Maximum Independent Set (MIS) sizes compared to other methods. Specifically, for graphs with $n = 200$ and $n = 300$, corresponding to 40,000 and 90,000 vertices respectively, LwD's performance suffered. This can be attributed to the need for increased episode lengths and time limits for LwD to effectively approximate the MIS size in such large graphs. Additionally, this method necessitates additional fine-tuning of hyperparameters to achieve results comparable to other methods.

As one of the main results of this paper, we highlight that the independent sets obtained for graphs $G_{n,n}$ with large $n$ values exhibit strong resemblance to Croft's construction \cite{croft1967incidence} for the lower bound of $m_1(\bbbr^2)$, which has remained unimproved since 1967. To illustrate this finding, we determined the maximal independent set size using the KaMIS solver (as the best solver in our experiments, see Table~\ref{tab1}) for the graph $G_{400,400}$ based on the torus $T_{l_1^*, l_2^*, \alpha^*}$ and present this comparison in Figure~\ref{kamis_vs_croft_comparison}.

\begin{figure}[ht]
\begin{tabular}{cc}
    {\includegraphics[scale=0.18]{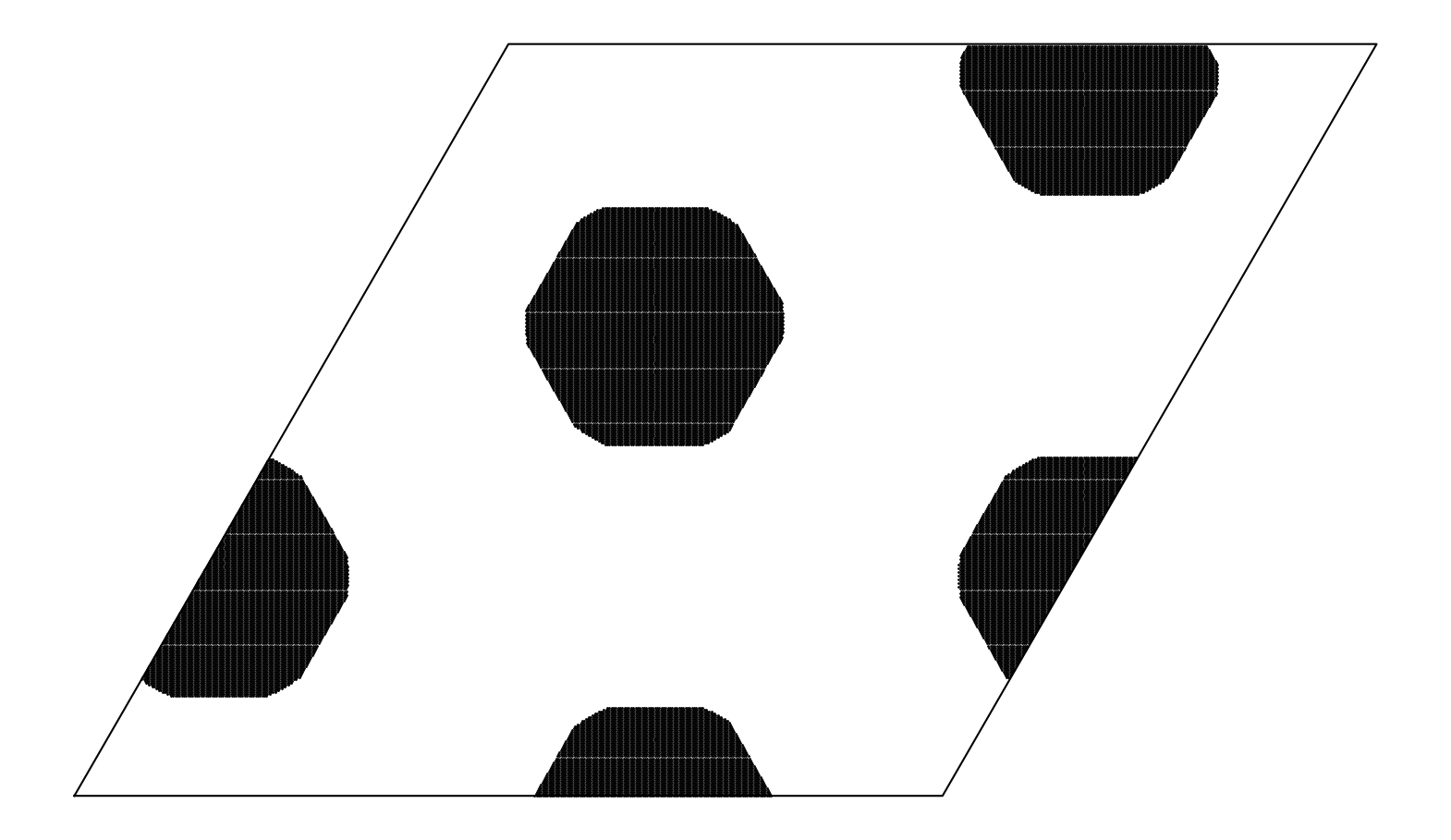}} &
    {\includegraphics[scale=0.13]{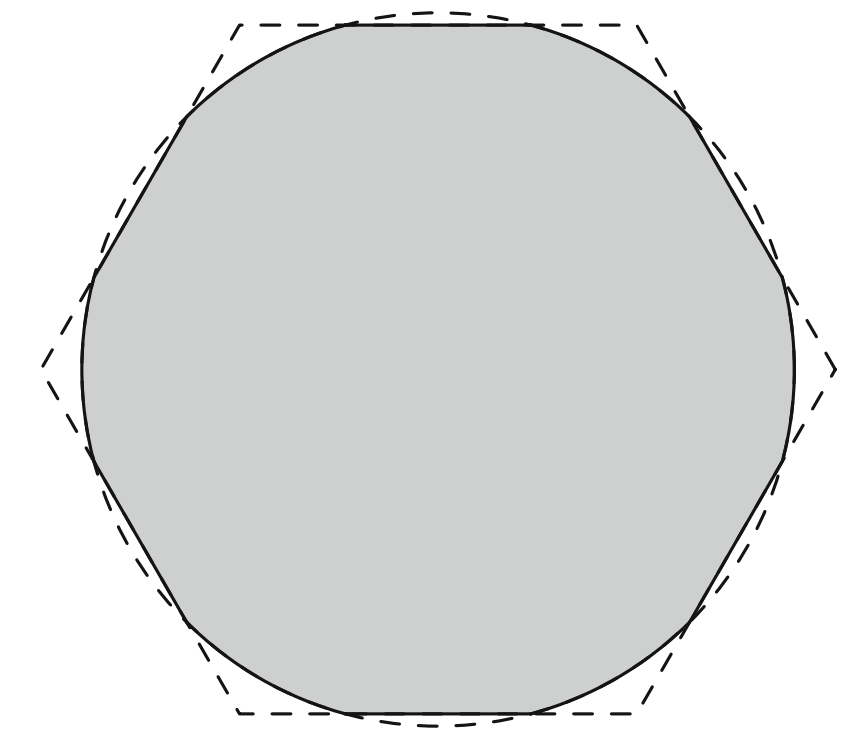}} 
\end{tabular}
\caption{On the \textit{left}, independent set via KaMIS: $n = 400$, $|M| = 35936$, $m_1(\bbbr^2) \ge 0.2246$ for the graph based on torus $T_{l_1^*, l_2^*, \alpha^*}$. On the \textit{right}, the optimal tortoise in Croft's construction \cite{croft1967incidence}.}
\label{kamis_vs_croft_comparison}
\end{figure}

Another intriguing aspect involves examining the local minimum obtained from Dataset-1 (refer to Table~\ref{tab1}). We examined the independent sets and associated lower bounds corresponding to the local minimum point on Dataset-1. This local minimum point was determined based on the average independent set size obtained from the four methods, or alternatively, using the KaMIS results. The independent sets on the flat torus with corresponding local minimum graph parameters were found via KaMIS. Results of this experiment are illustrated in Figure~\ref{min_density_examples}.

\begin{figure}
\begin{tabular}{cc}
    {\includegraphics[width=0.4\textwidth]{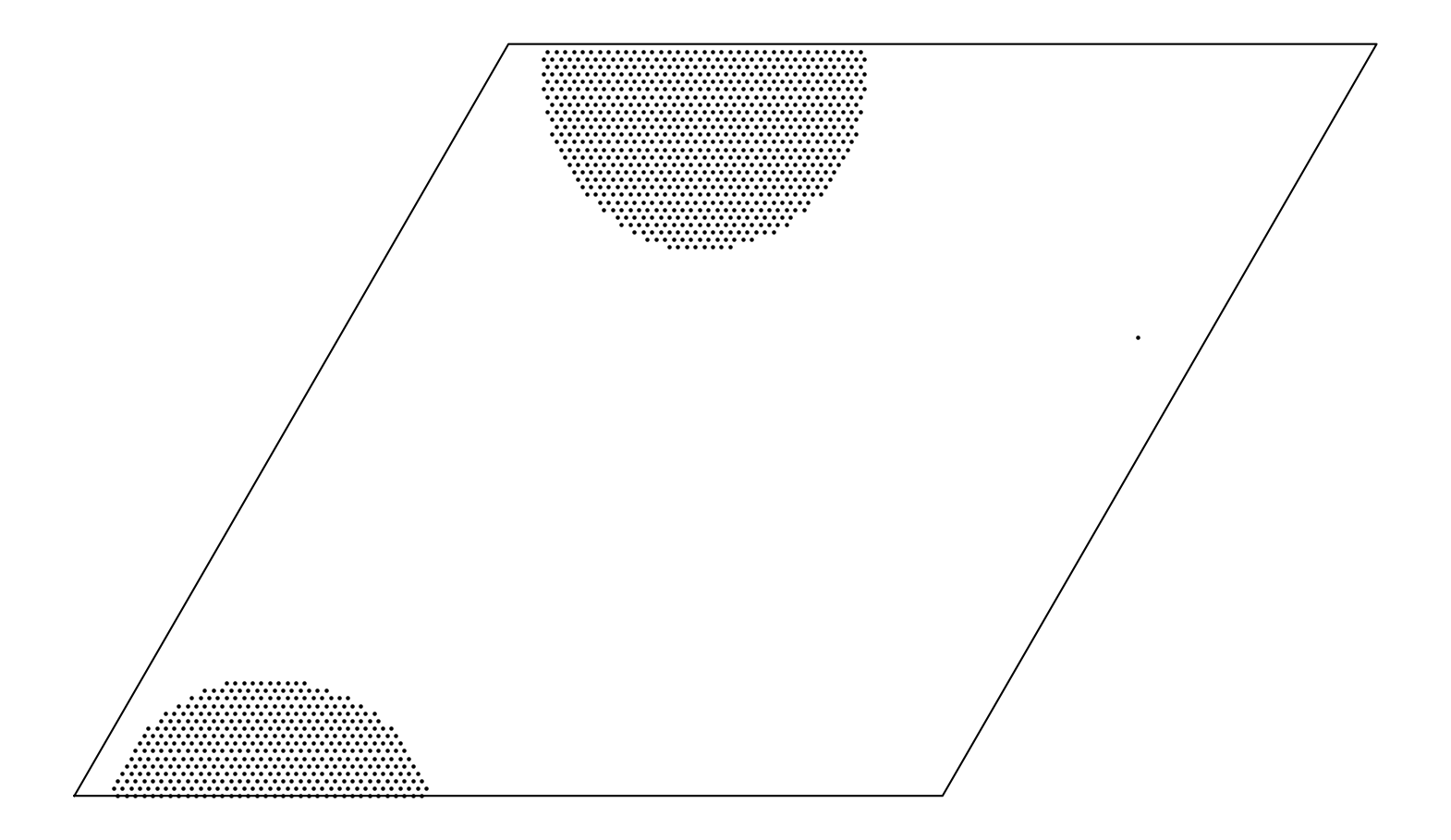}} &
    {\includegraphics[width=0.55\textwidth]{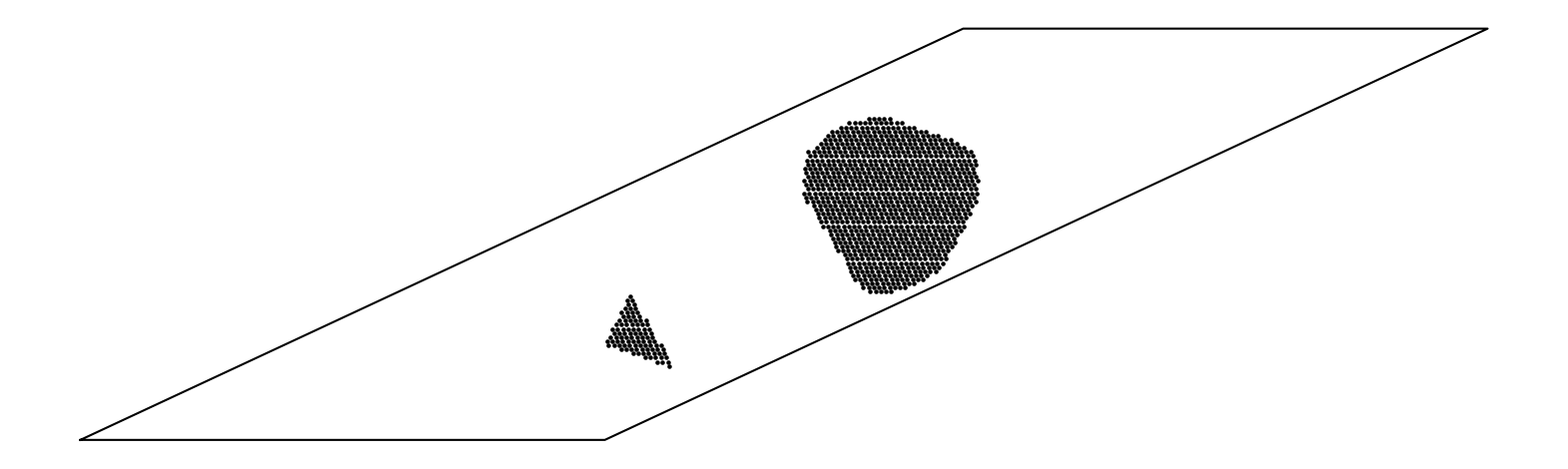}} \\
    $l_1 = 2.6, l_2 = 2.6, \alpha = \frac{\pi}{3}$  & $l_1 = 2.8, l_2 = 5.2, \alpha = \frac{5\pi}{36}$ \\
    $n = 100, |M| = 1273,$  & $n = 100, |M| = 1150$ \\
    $m_1(\bbbr^2) \ge 0.1273$ & $m_1(\bbbr^2) \ge 0.1150$  
\end{tabular}
\caption{On the \textit{left}, the local minimum was determined as the average (across four methods) of the maximum independent set size for those graphs where methods returned non-zero MIS sizes. On the \textit{right}, the local minimum was obtained using the KaMIS results for graphs where KaMIS found a non-empty independent set.}
\label{min_density_examples}
\end{figure}

\section{Conclusion}

This paper presents an approach for exploring lower bound estimates of $m_1(\bbbr^2)$ over the class of 1-avoiding planar sets that exhibit periodicity along two non-collinear vectors. The proposed reformulation of $m_1(\bbbr^2)$ as a Maximal Independent Set (MIS) problem on a specialized graph constructed on the flat torus is both theoretically and practically substantiated. The developed approach and programming tools \cite{our_repo} are also applicable to other problems related to the foundations of periodic sets with additional constraints. Notably, our approach can be used to approximate the maximal independent set in graphs constructed via an arbitrary flat torus $T_{l_1, l_2, \alpha}$.

If we restrict ourselves to the range of parameters under consideration, then the maximum independent set problem on a periodic grid leads to the approximations of Croft's construction (see Fig.~\ref{kamis_vs_croft_comparison}) and cannot improve the related lower bound $m_1(\bbbr^2) \ge 0.22936...$. Generally speaking, based on our experiments, we cannot claim that the estimation cannot be improved in this way for other parameter values.

Comparison of different algorithms for finding independent sets in geometric graphs of this type has shown that algorithms using GPU computation are less efficient than KaMIS in this case. Apparently, a neural network pre-trained on a graph with a large number of vertices is required for efficient use of GPU.

\begin{credits}
\subsubsection{\ackname} The author thanks Vsevolod Voronov for his guidance and discussions during this research.
\end{credits}

\bibliographystyle{splncs04}
\bibliography{ms}

\end{document}